\documentclass[11pt, reqno]{amsart}
\usepackage[margin=1in]{geometry}
\usepackage{amsmath,amssymb}
\usepackage{xcolor} % A package to add color.
\usepackage{amsthm}
\usepackage{amsfonts}
\usepackage{subcaption}
\usepackage{graphicx}
\usepackage[dvpis]{epsfig}
 \usepackage[colorlinks=true]{hyperref}
 \hypersetup{urlcolor=blue, citecolor=red}
\usepackage{multirow}

\makeatletter
\newcommand*{\rom}[1]{\expandafter\@slowromancap\romannumeral #1@}
\makeatother
\newcommand{\N}{\mathbb{N}}
\newcommand{\R}{\mathbb{R}}

\newtheorem{theorem}{Theorem}[section]
\newtheorem{corollary}[theorem]{Corollary}
\newtheorem{lemma}[theorem]{Lemma}
\newtheorem{proposition}[theorem]{Proposition}

\theoremstyle{definition}
\newtheorem{definition}[theorem]{Definition}
\newtheorem{example}[theorem]{Example}
\theoremstyle{remark}

\newtheorem{remark}[theorem]{Remark}
\numberwithin{equation}{section}

\newcommand{\T} {\mathbb T}
\newcommand{\pa}{\partial}

\newcommand{\e}{\varepsilon}
\newcommand{\lt}{\left}
\newcommand{\rt}{\right}
\newcommand{\bq}{\begin{equation}}
\newcommand{\eq}{\end{equation}}

\newcommand{\bbI}{\mathbb I}

\newcommand{\bbS}{\mathbb S}

\newcommand{\calA}{\mathcal A}

\newcommand{\calC}{\mathcal C}

\newcommand{\calF}{\mathcal F}

\newcommand{\calI}{\mathcal I}

\newcommand{\calR}{\mathcal R}

\newcommand{\intr}{\int_{\R^N}}
\renewcommand{\geq}{\geqslant}
\renewcommand{\leq}{\leqslant}
\newcommand{\dx}{\textnormal{d}x}
\newcommand{\dv}{\textnormal{d}v}
\newcommand{\ds}{\textnormal{d}s}
\newcommand{\du}{\textnormal{d}u}
\newcommand{\dr}{\textnormal{d}r}

\newcommand{\dA}{\textnormal{d}A}

\newcommand{\dsi}{\textnormal{d}\sigma}
\newcommand{\ov}{\overline}

\usepackage{soul}
\usepackage[normalem]{ulem}
\usepackage{cancel}

\renewcommand{\sout}[1]{}
\renewcommand{\cancel}[1]{}

\begin{document}

\title[From nonlinear Fokker--Planck to barotropic Euler equations]{Hydrodynamic limit from nonlinear Fokker--Planck to barotropic Euler equations}

\author[Carrillo]{Jos\'{e} A. Carrillo}
\address[Jos\'{e} A. Carrillo]{
    \newline Mathematical Institute, University of Oxford, Oxford OX2 6GG, UK}
\email{jose.carrillo@maths.ox.ac.uk}

\author[Koo]{Dowan Koo}
\address[Dowan Koo]{ \newline Mathematical Institute, University of Oxford, Oxford OX2 6GG, UK}
\email{dowan.koo@maths.ox.ac.uk}

\date{\today}


\thanks{\textbf{Acknowledgment.} 
JAC was supported by the Advanced Grant Nonlocal-CPD (Nonlocal PDEs for Complex Particle Dynamics: Phase Transitions, Patterns and Synchronization) of the European Research Council Executive Agency (ERC) under the European Union Horizon 2020 research and innovation programme (grant agreement No. 883363) and partially supported by the EPSRC EP/V051121/1. JAC was partially supported by the ``Maria de Maeztu'' Excellence Unit IMAG, reference CEX2020-001105-M, funded by MCIN\slash AEI \slash10.13039\slash501100011033\slash. 
 DK was supported by NRF grant no. 2022R1A2C1002820, RS-2024-00406821, and RS-2025-02312778.}

\begin{abstract} The hydrodynamic limit to the barotropic Euler equations, including power-law pressure $P(\rho)=\rho^\gamma$, for a kinetic nonlinear Fokker--Planck equation with degenerate diffusion is established. This extends the well-known result of the derivation of isothermal Euler equations via Fokker--Planck equation with linear diffusion. We establish the asymptotic analysis using the relative entropy method by quantifying error estimates for pressures and employing the generalized Log-Sobolev inequality for degenerate diffusion.
\end{abstract}
% ----------------------------------------------------------------

\maketitle

% ----------------------------------------------------------------

\section{Introduction}

We consider a kinetic nonlinear Fokker--Planck equation of the following form:
\bq\label{main_eq}
 \partial_t f + v \cdot \nabla_x f  = \frac{1}{\e}\nabla_v \cdot \lt( \nabla_v L_{\psi}(f) + (v-u_f)f\rt)
\eq
where $L_\psi$ is a non-linear diffusion law (see Definition \ref{def:L}). The distribution $f$ stands for the statistical density of a rarefied gas subject to the collisional operator on the right-hand side of \eqref{main_eq} and the corresponding physical quantities, namely the macroscopic density and momentum are defined respectively as
\[
\rho_f:=\intr f\,\dv,\qquad m_f:= \intr vf\,dv.
\]
The bulk velocity $u_f$ is defined as
\[
u_f:= \begin{cases} \frac{m_f}{\rho_f} \quad &\rho_f>0, \\
 0\quad &\rho_f=0.
\end{cases}
\]
The main objective of this work is to establish that the macroscopic quantities $\rho_f, u_f$ satisfy the compressible Euler equations with a certain pressure law as $\e \to 0$; hence, a broad class of barotropic Euler equations can arise as the hydrodynamic limit of \eqref{main_eq}.

One of the most typical barotropic Euler equations is the compressible fluid equations with a power-law pressure:
\begin{align*}
\begin{cases}
\pa_t \rho + \nabla_x \cdot(\rho u) = 0, \\
\pa_t (\rho u) + \nabla_x \cdot (\rho u \otimes u) + \theta \nabla_x \rho^\gamma = 0,
\end{cases}
\end{align*}
where $\gamma \ge 1$  and $\theta \ge 0$ are constants. Here, $\rho$ and $u$ denote the density and velocity of the gas. In particular, the $\theta=0$ case corresponds to the pressureless gas dynamics. As long as $\theta>0$, the system corresponds to the isothermal Euler for $\gamma=1$ and isentropic Euler for $\gamma>1$.

For the isothermal case, Berthelin and Vasseur \cite[Theorem 1.1]{BV05} established via the relative entropy method that the isothermal Euler equations can be derived as the hydrodynamic limit of the following kinetic equations:
\bq\label{eq:lFP}
\pa_t f + v \cdot \nabla_x f = \frac{1}{\e}\nabla_v \cdot (\theta \nabla_v f +(v-u_f)f).
\eq 
We remark that the linear diffusion in \eqref{eq:lFP} corresponds to the isothermal pressure law $P(\rho)=\theta \rho$.

For the isentropic case, Bouchut \cite{bouchut1999} introduced a variant of the BGK model to derive isentropic Euler equations for the range $\gamma \in (1,\frac{N+2}{N}]$, where $N\in\N$ denotes the dimension of the spatial domain. A mathematical justification of the hydrodynamic limit for this BGK model was established in \cite{BV05, KS24} based on the relative entropy method. This BGK operator has been recently employed in deriving the  Euler alignment system with isentropic pressure \cite{CH24b}. In the monodimensional case, the hydrodynamic limit can also be achieved even when shocks appear \cite{BB02a, BB02b}. 

Recently, the kinetic Fokker--Planck equation \eqref{main_eq} with $L_\psi(s)=s^m$ and without the alignment term $\nabla_v\cdot((v-u_f)f)$ was considered by Brigatti, Carlier, and Dolbeault \cite{BCD26}. The fundamental solutions are obtained in the range $m \in (1-\frac{1}{N},1) \cup (1,1+\frac{1}{N})$. The asymptotic behaviour, and the construction of weak solutions is studied under initial data being trapped between two self-similar profiles with different masses.

While the hydrodynamic limit related to the isothermal pressure or pressureless case via \eqref{eq:lFP} has been extensively studied \cite{MV08, KMT15, CCJ21, CJ21, CK23, FK19}, there has been no study on the derivation of isentropic or more general pressure via Fokker--Planck type equation, to the best of the authors' knowledge. This motivates the current study, and we fill this gap by considering \eqref{main_eq}. The existence of the system \eqref{main_eq} is not treated in this work; thus, the hydrodynamic limit result is a priori. However, more recently, the gradient flow structure and the construction of solutions using the JKO scheme for the Fokker--Planck equation \eqref{main_eq} with $L_\psi(s)=s^m$ with $u_f=0$ are studied in \cite{BCDQ26}.

The main technical contribution of this paper is Section \ref{sec:key}, where the errors of pressure tensors are controlled by (a lower bound of) dissipation of the kinetic models \eqref{main_eq}. This type of estimate first arises in the work of Berthelin and Vasseur \cite{BV05} in the case of BGK-type model with $\gamma = \frac{N+2}{N}$ case. Later, it was extended to $\gamma \in (1,\frac{N+2}{N})$ in \cite{KS24} by introducing \emph{extension map}. This extension map was recently employed to study nonlinear diffusion limit in \cite{CKT26}. In the current work, we carefully design the extension map to study general nonlinear diffusion. Then, combined with the generalized Log-Sobolev inequality obtained in \cite{CJMTU01} with relative entropy method, we establish the hydrodynamic limit result, which is addressed in Theorem \ref{thm:hdr}.

The rest of the paper is organized as follows. Section \ref{sec:P} introduces the axiomatic assumptions of diffusion law in kinetic model  \eqref{main_eq} with related properties. The relationship between diffusion law in kinetic model and pressure law in compressible fluid equation is addressed. In particular, a class of admissible pressure laws -- including power law pressures and even power with different exponents at end points -- are discussed in Section \ref{sec:ap}. Then, Section \ref{sec:M} studies the properties of the associated local Maxwellian to \eqref{main_eq} and obtain a lower bound of kinetic dissipation by the generalized Log-Sobolev inequality. Section \ref{sec:key} records the key technical lemmas to establish the hydrodynamic limit. Lastly, in Section \ref{sec:mr}, we state our main results and provide their proof based on the relative entropy method. In Appendix \ref{app}, we check our diffusion laws satisfy the conditions for the generalized Sobolev inequality to hold.

\textbf{Notations.} Throughout this paper, $C$ may denote a generic constant, which may differ from line to line. We sometimes put subscripts to stress the dependence on certain parameters.

\textbf{Domain.} We consider $\R^N$ as a spatial domain. The parallel results should also be satisfied for the periodic domain $\T^N$.

\section{Kinetic diffusion -- macroscopic pressure correspondence}\label{sec:P}
In this section, we describe how the macroscopic pressure law at fluid level corresponds to the microscopic nonlinear diffusion law at the kinetic level. We begin by defining the non-linear diffusion law $L_\psi$ in \eqref{main_eq} as follows.

\begin{definition}\label{def:L}
Let $\psi:[0,\infty)\to[0,\infty)$ be a $ \calC^3$ function and define
\[
\Psi(s):=\int_0^s \psi(u)\,du,\qquad 
L_\psi(s):=s\psi(s)-\Psi(s)\qquad (s\ge0).
\]
We call $L_\psi$ the \emph{diffusion law} and $\Psi$ the associated entropy. We assume:
\begin{itemize}
\item[(A1)] $\psi'(s)>0$ for all $s>0$ (hence $\psi$ is strictly increasing on $(0,\infty)$).
\item[(A2)] $\psi(0)=0$.
\item[(A3)] There exists $\Lambda\ge1$ such that for all $0\le s_1<s_2<2s_1$,
\begin{equation}\label{eq:double}
\frac{1}{\Lambda}(\psi^{-1})'(s_1)\le (\psi^{-1})'(s_2)\le \Lambda(\psi^{-1})'(s_1).
\end{equation}
\item[(A4)] $\psi'(s) + s\psi''(s) \ge 0$ for all $s>0$.
\end{itemize}
\end{definition}

\begin{example}

A typical class of degenerate diffusion satisfying (A1) and (A2) is porous medium diffusion, namely,
\[
L_{\psi}(s)=s^m,\quad \Psi(s):=\frac{s^m}{m-1}
\]
with $m>1$. Then we can easily notice that $\psi(s) = \frac{ms^{m-1}}{m-1}$ satisfies
\[
\psi(0)=0,\quad \psi'(s) = ms^{m-2}>0
\]
for any $s>0$. Moreover, it also satisfies (A3). Since $\psi^{-1}(s) = \lt(\frac{m-1}{m}s\rt)^{\frac{1}{m-1}}$, we have
\[
(\psi^{-1})'(s)= \frac{1}{m-1}\lt(\frac{m-1}{m}\rt)^{\frac{1}{m-1}}s^{\frac{2-m}{m-1}}.
\]
For $m\ge2$, $(\psi^{-1})'$ is nonincreasing, hence \eqref{eq:double} holds with $\Lambda_m = 2^{\frac{m-2}{m-1}} \ge 1$. For $1<m<2$, $(\psi^{-1})'$ is increasing and \eqref{eq:double} is satisfied with $\Lambda_m= 2^{\frac{2-m}{m-1}}>1$. Notice that $\Lambda_m \nearrow \infty$ as $m \to 1^+$. Lastly, $\psi' + s\psi'' = m(m-1) s^{m-2} > 0$ for all $s>0$ confirms (A4). 
\end{example}

\begin{remark}
Assumption (A1) implies $\psi$ is invertible on $(0,\infty)$ and $\Psi$ is convex. Assumption (A2) ensures degeneracy of diffusion. Thus, it excludes the linear diffusion $L_\psi(s)=s$ with the Boltzmann entropy $\Psi(s)=s\log s$, which is in line with the previous observation $\Lambda_m \to \infty$ as $m \to 1^+$. Assumption (A3) is a local doubling condition on $(\psi^{-1})'$, which will be crucially employed in proving  Proposition \ref{prop:diss} and Lemma \ref{lem:crit}, which are to come up later. (A4) will be imployed to satisfy the assumption on nonlinear diffusion $L_\psi$ for the generalized Log-Sobolev inequality in Section \ref{app}.
\end{remark}

We then introduce an important quantity $h_\psi$, which plays the role of enthalpy at the macroscopic level.

\begin{lemma}\label{lem:h}
Suppose $\psi$ satisfies (A1)--(A2). There exists a strictly increasing $h_\psi\in C^1((0,\infty))$ with $h_\psi(0)=0$ such that for every $\rho\ge0$
\begin{equation}\label{def:h}
\int_{\mathbb R^N} \psi^{-1}\!\Big( \big(h_\psi(\rho)-\tfrac12|v|^2\big)_+\Big)\,dv = \rho.
\end{equation}
Moreover, for every $\rho>0$, $c>1$,
\begin{equation}\label{eq:h:dou}
h_\psi(c\rho) \le c^{\frac{2}{N}}h_\psi(\rho).
\end{equation}
As a consequence,
\bq\label{eq:h:asym}
    \begin{aligned}
    h_\psi(\rho) \le h_\psi(1)\rho^{\frac{2}{N}}\quad \text{for}\,\,\, \rho> 1,\\
     h_\psi(\rho) \ge h_\psi(1)\rho^{\frac{2}{N}}\quad \text{for}\,\,\, \rho< 1.
    \end{aligned}
    \eq
If we further assume $\psi$ satisfies
\bq\label{eq:d:ov}
\psi^{-1}(2s) \le \ov\Lambda \psi^{-1}(s)\quad\text{for all}\,\,\,s\ge 0,
\eq
for some $\ov\Lambda >1$, then we have
\bq\label{eq:h:lb}
h_\psi(\rho) \ge 2h_\psi\lt(\tfrac{\rho}{L}\rt)
\eq
where $L:=2^{\frac{N}{2}}\ov\Lambda^2>1$.
\end{lemma}
\begin{remark}
    The estimate \eqref{eq:h:lb} cannot be obtained from \eqref{eq:h:dou}, thus we impose the growth condition \eqref{eq:d:ov}. Moreover, the (A3) condition, together with (A1), (A2), implies that $\psi$ verifies \eqref{eq:d:ov} with $\ov\Lambda = 2\Lambda$. Indeed,
    \[
    \psi^{-1}(2s)=\int_0^{2s}(\psi^{-1})'(u)\,\du = 2\int_0^s (\psi^{-1})'(2u)\,\du \le 2\Lambda \int_0^s (\psi^{-1})'(u)\,\du =2\Lambda \psi^{-1}(s).
    \]
\end{remark}

\begin{proof}[Proof of Lemma \ref{lem:h}] We first investigate the well-definedness of $h_\psi$. To this end, for each $R \ge 0$, we set
\[
\calI(R):= \intr \psi^{-1} \lt(\lt( R - \frac{1}{2}|v|^2\rt)_+\rt)\,\dv.
\]
By polar coordinate, we can rewrite
\bq\label{eq:IR}
\calI(R)= \omega_N\int_0^{\sqrt{2R}} \psi^{-1}\lt(R- \frac{1}{2}r^2\rt)r^{N-1}\,\dr,
\eq
where $\omega_N=|\bbS_{N-1}|$.
Note that $\calI(0)=0$ and $\calI:[0,\infty)\to[0,\infty)$ is strictly increasing and continuously differentiable since
\[
\calI'(R)= \omega_N\int_0^{\sqrt{2R}} (\psi^{-1})'\lt(R- \frac{1}{2}r^2\rt)r^{N-1}\,\dr>0,
\]
for any $R>0$. This provides that $\calI:[0,\infty) \to [0,\infty)$ is injective and $\calI \in C^1((0,\infty))$. For $R \ge 1$, as $\psi$ is increasing due to (A1), we have
\[
\calI(R) \ge \int_{B_{\sqrt{R}}(0)} \psi^{-1}\lt(\tfrac{R}{2}\rt)\,\dv \ge \omega_N\psi^{-1}\lt(\tfrac{1}{2}\rt)R^{\frac{N}{2}} \to +\infty
\]
as $R$ tends to infinity. This suggests that $\calI$ is surjective, hence, for each $\rho \ge 0$, we can uniquely define $h_\psi(\rho)$ by $h_\psi(\rho):=\calI^{-1}(\rho)$, this proves the first statement.

To obtain a doubling type estimate \eqref{eq:h:dou}, we define for each $c>1$, $R>0$:
\bq\label{eq:Ci}
\begin{aligned}
\calC_i(c,R):=&\lt\{v \in \R^d: \tfrac{R}{c^{i+1}}\le R - \tfrac{1}{2}|v|^2 \le \tfrac{R}{c^i}  \rt\}\\
=& \lt\{v \in \R^d: R\lt(1-\tfrac{1}{c^{i}}\rt)\le\tfrac{1}{2}|v|^2 \le R\lt(1-\tfrac{1}{c^{i+1}}\rt)  \rt\},
\end{aligned}
\eq
where $i=0,1,2,\cdots$.
We notice that $v \in \calC_1(c,cR)$ if and only if  $v \in c^{\frac{1}{2}}\calC_1(c,R)$ by self similarity. We thus deduce that $|\calC_i(c,cR)|= |c^{\frac{1}{2}}\,\calC_i(c,R)| =c^{\frac{N}{2}}|\calC_i(c,R)|$. We then estimate 
\bq\label{eq:est_diadic}
\begin{aligned}
    \calI(cR) &= \sum_{i=0}^\infty \int_{\calC_i(c,cR)} \psi^{-1}((cR - \tfrac{1}{2}|v|^2)_+)\,\dv \ge \sum_{i=0}^\infty \psi^{-1}\lt(\tfrac{R}{c^i}\rt)|\calC_i(c,cR)|\\
    &=c^{\frac{N}{2}}\sum_{i=0}^\infty \psi^{-1}\lt(\tfrac{R}{c^i}\rt)|\calC_i(c,R)|
    = c^{\frac{N}{2}}\sum_{i=0}^\infty \int_{\calC_i(c,R)} \psi^{-1}\lt(\tfrac{R}{c^i}\rt)\,\dv  \ge c^{\frac{N}{2}}\calI(R).
\end{aligned}
\eq
Since $\calI$ is an inverse of $h_\psi$ and $h_\psi$ is increasing as well, we have $cR \ge h_\psi(c^\frac{N}{2}R)$. Putting $R= h_\psi(\rho)$ yields $ch_\psi(\rho) \ge h_\psi(c^{\frac{N}{2}}\rho)$. Then, switching $c:=c_*^{\frac{2}{N}}>1$ gives $h_\psi(c_*\rho)\le c_*^{\frac{2}{N}}h_\psi(\rho)$, which is equivalent to \eqref{eq:h:dou}.

To investigate the asymptotic bounds of $h_\psi$, we first consider $\tilde\rho > 1$. Inserting $(c,\rho)=(\tilde \rho, 1)$ into \eqref{eq:h:dou} gives $h_\psi(\tilde \rho) \le h_\psi(1){\tilde \rho}^{\frac{2}{N}}$. On the other hand, if $\tilde \rho < 1$, we take $(c,\rho)=(\tilde \rho^{-1},\tilde\rho)$ into \eqref{eq:h:dou} to deduce $h_\psi(1) \le (\tilde \rho)^{-\frac{2}{N}}h_\psi(\tilde \rho)$. This establishes \eqref{eq:h:asym}.

Finally, under \eqref{eq:d:ov}, we estimate $\calI$ similarly as \eqref{eq:est_diadic}:
\[
\begin{aligned}
\calI(2R) = \sum_{i=0}^\infty \int_{\calC_i(2,2R)}\psi^{-1}((2R-\tfrac{1}{2}|v|^2)_+)\,\dv &\le \sum_{i=0}^\infty \psi^{-1}\lt(\tfrac{2R}{2^i}\rt)|\calC_i(2,2R)|\\
    &\le2^{\frac{N}{2}}\Lambda^2\sum_{i=0}^\infty \psi^{-1}\lt(\tfrac{R}{2^{i+1}}\rt)|\calC_i(2,R)| \le 2^{\frac{N}{2}}\Lambda^2\calI(R).
\end{aligned}
\]
Similarly as above, we set $R=h_\psi(\rho)$ after taking $h_\psi$ on both sides to find that
\[
2h_\psi(\rho) \le h_\psi(2^{\frac{N}{2}}\Lambda^2 \rho).
\]
We then normalize $\rho = \frac{\tilde \rho}{2^{\frac{N}{2}}\Lambda^2}$ to obtain \eqref{eq:h:lb}. This concludes the proof.
\end{proof}

Based upon the macroscopic enthalpy $h_\psi$,  we define the macroscopic pressure and entropy as:
\begin{equation}\label{eq:P,Phi}
P_\psi(\rho):=\rho h_\psi(\rho)-\Phi_\psi(\rho),\qquad 
\Phi_\psi(\rho):=\int_0^\rho h_\psi(s)\,ds.
\end{equation}
\begin{remark}
Indeed, the relation between $\psi^{-1}$ and $h_\psi$ is equivalent to the equilibrium profile at kinetic level and the diffusion law at macrscopic law in  \cite{DMOS07} for the diffusion limit.
\end{remark}
We then deduce the elementary behavior of macroscopic pressure and entropy.

\begin{corollary}\label{cor:sh(s)}
Under (A1)--(A2), we have $\Phi_\psi:[0,\infty)\to [0,\infty)$ is a strictly increasing and convex $\calC^2$ function satisfying $\Phi_\psi(0)=0$. $P_\psi:[0,\infty)\to [0,\infty)$ is a $\calC^1$ strictly increasing function satifying $P_\psi(0)=0$. Moreover, it holds that $\rho h_\psi(\rho) \ge \Phi_\psi(\rho)$ and
\[
P_\psi(\rho) \le \rho h_\psi(\rho) \le (2^{\frac{2}{N}+1}) \Phi_\psi(\rho).
\]
If $\psi$ additionally verifies \eqref{eq:d:ov}, then 
\[
P_\psi(\rho) \ge \tfrac{\rho}{L} h_\psi\lt(\tfrac{\rho}{L}\rt),
\]
with $L=2^{\frac{N}{2}}\ov \Lambda^2>1$.

\end{corollary}

\begin{proof}
The qualitative properties of $\Phi_\psi$ and $P_\psi$ are immediate from the formula \eqref{eq:P,Phi} and Lemma \ref{lem:h}. It is also immediate that $0\le \Phi_\psi(\rho)=\int_0^\rho h_\psi(s)\,\ds \le \rho h_\psi(\rho)$, $P_\psi(\rho)\le \rho h_\psi(\rho)$.

To investigate an upper bound of $\rho h_\psi(\rho)$, we split $\Phi_\psi(\rho)$ in a dyadic manner and apply Lemma \ref{lem:h} to see that
\[
\int_0^\rho h_\psi(s)\,ds = \sum_{k=1}^\infty \int_{\rho/2^k}^{\rho/2^{k-1}} h_\psi(s)\,ds
\ge \sum_{k=1}^\infty \frac{\rho}{2^k}\,h_\psi\!\big(\tfrac{\rho}{2^k}\big).
\]
By recursively applying \eqref{eq:h:dou}, we have $h_\psi(\rho/2^k)\ge 2^{-\frac{2k}{N}} h_\psi(\rho)$, which yields
\[
\Phi_\psi(\rho)\ge \sum_{k=1}^\infty \frac{\rho}{2^k}\cdot \frac{h_\psi(\rho)}{2^{\frac{2k}{N}}}
= \frac{\rho h_\psi(\rho)}{2^{\frac{2}{N}+1}-1}.
\]
 Lastly, under \eqref{eq:d:ov}, we recall \eqref{eq:h:lb} to find that
 \[
 P_\psi(\rho) = \int_0^\rho (h_\psi(\rho)-h_\psi(s))\,\ds \ge \int_0^{\frac{\rho}{L}}(h_\psi(\rho) - h_\psi(s))\,\ds \ge \frac{\rho}{L} (h_\psi(\rho)-h_{\psi}(\tfrac{\rho}{L})) \ge \tfrac{\rho}{L} h_\psi\lt(\tfrac{\rho}{L}\rt)
 \]
since $h_\psi$ is increasing as desired. 
\end{proof}

For the readers convenience, we conclude this section by presenting the variables concerning the kinetic and fluid laws.

\begin{table}[h]%\label{table1}
\centering
\begin{tabular}{|c||c|c|c|}
\hline
 & Diffusion/Pressure & Enthalpy & Entropy\\ \hline
Kinetic & $L_\psi$ & $\psi$ &$\Psi$ \\ \hline
Fluid & $P_\psi$ & $h_\psi$ & $\Phi_\psi$ \\ \hline
\end{tabular}
\caption{Mesoscopic and macroscopic laws.}
%\label{table}
\end{table}

\section{Admissible Pressures}\label{sec:ap}

This section concerns  concrete examples of macroscopic pressures $P_\psi(\rho)$ that arise from $\psi$ satisfying (A1)--(A3). From Lemma \ref{lem:h} and Corollary \ref{cor:sh(s)}, we have
\[
P_\psi(\rho) \lesssim \rho^{1+\frac{2}{N}}\quad \text{for large}\,\,\,\rho,\quad P_\psi(\rho)\gtrsim \rho^{1+\frac{2}{N}}\quad \text{for small}\,\,\, \rho.
\]
We establish that a broad class of pressure laws are indeed admissible in the class of $P_\psi$. We first investigate the case of power law pressures.

\begin{lemma}[Power law pressure]\label{prop:pl}
    Let $m>1$ and $\theta_m>0$ be given with $L_{\psi_m}(s):=\theta_m s^m$. Then we have
    \[
    P_{\psi_m}(\rho)=\theta \rho^\gamma
    \]
    where
    \bq\label{eq:gm}
    \gamma =\gamma(m) := 1 + \lt(\frac{1}{m-1}+\frac{N}{2}\rt)^{-1},
    \eq
    with $\theta>0$ satisfying \eqref{eq:theta}. Conversely, for any given $\gamma \in (1,1+ \frac{2}{N})$ and  $\theta>0$, there exist unique $m$ and $\theta_m$, explicitly given as \eqref{eq:m_g}--\eqref{eq:theta_m}, such that $P_{\psi_m}(\rho)=\theta \rho^\gamma$.
\end{lemma}

\begin{proof}
    Notice that $L_{\psi_m}(s)=\theta_m s^m$ is equivalent to $\psi_m(s) = \frac{m\theta_m}{(m-1)}s^{m-1}$, which gives
    \[
    \psi_m^{-1}(s)=\lt(\frac{m-1}{m\theta_m}\rt)^{\frac{1}{m-1}}s^{\frac{1}{m-1}}.
    \]
    We first characterize $h_{m}=h_{\psi_m}$ and determine the values of $m$ and $\theta_m$ to obtain $P_{\psi_m}(\rho)=\theta \rho^\gamma$.

    We recall \eqref{def:h} and use polar coordinate as \eqref{eq:IR} to find
    \[
    \rho = \omega_N \int_0^{\sqrt{2h_m(\rho)}} \lt(\frac{m-1}{m\theta_m}\rt)^{\frac{1}{m-1}}\lt(h_m(\rho) - \frac{1}{2}r^2\rt)^{\frac{1}{m-1}}r^{N-1}\,dr.
    \]
    By changing coordinate $r=\sqrt{2h_m(\rho)s}$, we have $r^{N-1}dr=\tfrac{1}{2}(2h_m(\rho))^{\frac{N}{2}}s^{\frac{N}{2}-1}\ds$, so that
    \[
    \rho = \frac{\omega_N}{2} \lt(\frac{m-1}{m\theta_m}\rt)^{\frac{1}{m-1}} (2h_m(\rho))^{\frac{N}{2}+\frac{1}{m-1}}\int_0^1(1-s)^{\frac{1}{m-1}}s^{\frac{N}{2}-1}\,\ds.
    \]
    We thus define $\gamma>1$ as in \eqref{eq:gm} so that $\frac{\gamma}{\gamma-1}=\frac{N}{2}+\frac{m}{m-1}$. By noticing $\omega_N=\frac{2\pi^{{N}/{2}}}{\Gamma(\frac{N}{2})}$ and
    \[
    \int_0^1 (1-s)^{\frac{1}{m-1}}s^{\frac{N}{2}-1}\,\ds = B\lt(\frac{m}{m-1},\frac{N}{2}\rt) = \frac{\Gamma(\frac{m}{m-1})\Gamma(\frac{N}{2})}{\Gamma(\frac{\gamma}{\gamma-1})},
    \]
    we find
    \[
    h_m(\rho)=\frac{1}{2}\rho^{\gamma-1}\lt[\frac{\pi^{\frac{N}{2}}\Gamma(\frac{m}{m-1})}{\Gamma(\frac{\gamma}{\gamma-1})}\rt]^{-(\gamma-1)} \lt(\frac{m\theta_m}{m-1}\rt)^{\frac{\gamma-1}{m-1}}.
    \]
    This shows that $P_{\psi_m}(\rho)=\theta \rho^\gamma$ for some $\theta$, where $\theta=\theta(m, \theta_m,N)$ is defined as
    \bq\label{eq:theta}
    \frac{2\theta \gamma}{\gamma-1}=\lt[\frac{\pi^{\frac{N}{2}}\Gamma(\frac{m}{m-1})}{\Gamma(\frac{\gamma}{\gamma-1})}\rt]^{-(\gamma-1)}\lt(\frac{m\theta_m}{m-1}\rt)^{\frac{\gamma-1}{m-1}}.
    \eq

    To conclude, we find $m>1$ and $\theta_m>0$ to satisfy $P_{\psi_m}=\theta\rho^\gamma$ for any given $\theta>0$, $\gamma \in (1,1+\frac{2}{N})$.  We uniquely choose $m>1$ satisfying \eqref{eq:gm}, in other words,
    \bq\label{eq:m_g}
    m:= 1 + \lt(\frac{1}{\gamma-1}-\frac{N}{2}\rt)^{-1}.
    \eq
    To determine $\theta_m>0$, we recall the so-called ``degree of freedom" to simplify the notation:
    \bq\label{eq:dfree}
    \frac{d}{2}:=\frac{1}{\gamma-1}-\frac{N}{2}.
    \eq
    Thus, \eqref{eq:theta} can be reformulted as
    \bq\label{eq:theta_m}
    \theta_m=\frac{2}{d+2} \lt[\lt(\frac{2\theta\gamma}{\gamma-1}\rt)^{\frac{1}{\gamma-1}} \frac{\pi^{N/2}\Gamma(\frac{d}{2}+1)}{\Gamma(\frac{\gamma}{\gamma-1})}\rt]^{\frac{2}{d}}.
    \eq
\end{proof}

\begin{remark}\label{rmk:psi_m}
    Throughout the paper, we always denote
    \[
    \psi_m(s):=\frac{m\theta_m}{m-1}s^{m-1}.
    \]
    If $\theta_m$ is given as \eqref{eq:theta_m}, then we can formally characterize its asymptotic limit as $m \to \infty$. By writing 
    \bq\label{eq:c1c2}
      c_{1,m} := \frac{2\gamma\theta}{\gamma-1},\quad c_{2,m} := \lt[\left(c_{1,m}\right)^{\frac{1}{\gamma-1}} \lt(\frac{\pi^{N/2}\Gamma\left(\frac{d}{2}+1\right)}{\Gamma(\frac{\gamma}{\gamma-1})}\rt)\rt]^{-1},
\eq
with $\gamma$, $d$ defined as \eqref{eq:gm} and \eqref{eq:dfree} respectively, 
  $\theta_m$ in \eqref{eq:theta_m} can be rewritten as
  \[
  \theta_m = \frac{m-1}{m}(c_{2,m})^{-\frac{1}{m-1}}.
  \]
As $m \to \infty$, one can notice that $P_{\psi_m}(\rho)=\theta\rho^{\gamma(m)} \to \theta \rho^{1+ \frac{2}{N}}$ and 
$c_{2,m} \to  c_{2,\infty}$ where $c_{2,\infty}=[(N+2)\theta]^{-\frac{N}{2}}\frac{\Gamma(\frac{N}{2}+1)}{\pi^{\frac{N}{2}}}$. In particular, since $\psi_m(s) = (\frac{s}{c_{2,m}})^{m-1}$, we have
    \[
    \psi_m(s) \to \psi_{\infty}(s):=\begin{cases}
    +\infty \quad &\text{if}\quad s> c_{2,\infty} \\
        \frac{1}{2} \quad &\text{if}\quad s= c_{2,\infty}\\
        0\quad &\text{if}\quad s< c_{2,\infty}.
        \end{cases}
    \]
\end{remark}

In fact, not only the power law pressure but pressure with different exponents depending on the density near the origin or at infinity are admissible in the class of $P_\psi$.

\begin{lemma}\label{lem:different}
    Let $1<\alpha <\beta <+\infty$. We define a class of functions $\calA_{\alpha,\beta}$ as
    \[
    \calA_{\alpha,\beta}:=\lt\{\psi \text{ verifies (A1)--(A2), \eqref{eq:d:ov} satisfying }(\text{S-}\alpha),\, (\text{S-}\beta) \rt\}
    \]
   where
   \[%\label{eq:s1}]
   \begin{aligned}
   &(\text{S-}\alpha)\quad \underline{c}s^{\frac{1}{\alpha-1}} \le \psi^{-1}(s) \le \ov{c}s^{\frac{1}{\alpha-1}}\quad\text{for all}\,\,\,s \le M_\alpha \\
   %\eq
   %\bq\label{eq:s2}
   &(\text{S-}\beta)\quad \underline{c}s^{\frac{1}{\beta-1}} \le \psi^{-1}(s) \le \ov{c}s^{\frac{1}{\beta-1}} \quad\text{for all}\,\,\,s \ge M_\beta,
   \end{aligned}
   \]
    for some $\ov{c} \ge \underline{c}>0$, with $0<M_\alpha <M_\beta$. Then, for any $\psi \in \calA_{\alpha, \beta}$, we have the asymptotic behaviour of $P_\psi(\rho)$ characterized as follows: if $P_\psi(\rho)$ is small (resp. large), 
    \[
    P_{\psi}(\rho) \sim \rho^{\gamma(\alpha)},\quad\text{resp.}\quad P_{\psi}(\rho) \sim \rho^{\gamma(\beta)},
    \]
    where $\gamma(\cdot)$ is defined as \eqref{eq:gm}.
\end{lemma}
\begin{proof}
First, we investigate a lower bound of $h_\psi(\rho)$. Since $h_\psi(\rho)$ is defined as an inverse function of $\calI$ in \eqref{eq:IR} and $\psi^{-1}$ is increasing thanks to (A1), we immediately have
\[
\rho = \omega_N\int_0^{\sqrt{2h_\psi(\rho)}}\psi^{-1}(h_\psi(\rho)-\tfrac{1}{2}r^2)r^{N-1}\,\dr \le \frac{\omega_N}{N}\psi^{-1}(h_\psi(\rho))(2h_\psi(\rho))^{\frac{N}{2}}.
\]
We apply upper bounds of $\psi^{-1}$ to obtain
\[
h_\psi(\rho) \gtrsim \rho^{\gamma(\alpha)-1}\quad \text{resp.}\quad h_\psi(\rho) \gtrsim \rho^{\gamma(\beta)-1}
\]
if $h_\psi(\rho)$ is small (resp. large).

On the other hand, to obtain a upper bound of $h_\psi(\rho)$, we recall \eqref{eq:Ci} and proceed similarly as in the proof of Lemma \ref{lem:h}. We observe that $h_\psi(\rho)$ satisfies
    \[
    \rho = \int \psi^{-1}((h_\psi(\rho)-\tfrac{1}{2}|v|^2)_+)\,\dv \ge \psi^{-1}(\tfrac{h_\psi(\rho)}{2})\lt|\calC_0\lt(2,h_\psi(\rho)\rt)\rt|,
    \]
with
$|\calC_0(2,R)|= \frac{\omega_N}{N}R^{\frac{N}{2}}$. This provides
\[
\rho \ge \frac{\omega_N}{N \ov{\Lambda}}\psi^{-1}(h_\psi(\rho))(h_\psi(\rho))^{\frac{N}{2}}.
\]
When the $h_\psi(\rho)$ is large (resp. small), we have
\[
h_\psi(\rho) \lesssim \rho^{\gamma(\beta)-1}\quad (\text {resp.}\,\,\,  h_\psi(\rho)\lesssim \rho^{\gamma(\alpha)-1}).
\]
Since $P_\psi(\rho) \sim \rho h_\psi(\rho)$ by Corollary \ref{cor:sh(s)}, smallness (resp. largeness) of $h_\psi(\rho)$ is equivalent to that of $P_\psi(\rho)$, and we deduce the asymptotic behaviour of pressures in these two extreme regimes, thus we conclude the proof.  
\end{proof}

\begin{remark}
    Obviously, the class of diffusion laws $\calA_{\alpha, \beta}$ defined in Lemma \ref{lem:different} is non empty as one can easily construct $\psi \in \calA_{\alpha,\beta}$ by connecting two smooth curves which behave as $s^{\alpha}$ near the origin and $s^{\beta}$ in the far field. 
\end{remark}

Lastly, in order to close the estimate for the relative entropy below in Lemma \ref{olive}, we further impose $h_\psi$ satisfy the following
\[
\text{(A5)}\quad \rho |h_\psi''(\rho)| \le C h_\psi'(\rho)\qquad \text{for all}\,\,\, \rho>0
\]%\end{equation}
for some $C>0$. It is evident to see that this inequality holds for $\psi(s)=\theta s^m$ with $m>1$. It also holds if $h_\psi''$ behaves as a power law of $\rho$ near the origin and the far field.

\section{Properties of the associated Maxwellians}\label{sec:M}
In this section, we study the properties of the local equilibria of the collision operator in \eqref{main_eq}:
\begin{equation}\label{eq:Q}
\mathcal Q_\psi(f):=\nabla_v\cdot\big(\nabla_v L_\psi(f) + (v-u_f)f\big).
\end{equation}
Set the admissible class
\begin{equation}\label{eq:A}
\mathcal A(\mathbb R^N):=\big\{f\in L^1_+(\mathbb R^N):\ |v|^2 f + \Psi(f)\in L^1(\mathbb R^N)\big\}.
\end{equation}
For fixed $f\in\mathcal A(\mathbb R^N)$, we consider the Cauchy problem
\begin{equation}\label{eq:ctrl}
\begin{cases}
\partial_t g = \nabla_v\cdot\big(\nabla_v L_\psi(g) + (v-u_f)g\big),\\
g(0,\cdot)=f(\cdot).
\end{cases}
\end{equation}
Denoting $V(v):=\tfrac12|v-u_f|^2$, the equation \eqref{eq:ctrl} can be rewritten as
\begin{equation}\label{eq:ct}
\partial_t g = \nabla_v\cdot\big(g\nabla_v(\psi(g)+V)\big).
\end{equation}
Define the free energy and information (dissipation) term as:
\bq\label{eq:F,I}
\mathcal F(g):=\int_{\mathbb R^N} \Psi(g) + V g\,dv,\qquad 
I(g):=\int_{\mathbb R^N} g\big|\nabla_v(\psi(g)+V)\big|^2\,dv.
\eq
Then, one can formally obtain the dissipation identity
\begin{equation}\label{eq:diss:g}
\frac{d}{dt}\mathcal F(g) = - I(g) \le 0.
\end{equation}
Classical theory (see e.g. \cite{CJMTU01}) provides existence and convergence toward its long-time asymptotic (hence, steady state) $g_\infty$; we state several standard results below.

\begin{lemma}\label{lem:a}
Let $g_\infty$ be the long-time limit of \eqref{eq:ctrl}. Then, we have
\begin{enumerate}
\item (Mass and momentum) \[
\int_{\mathbb R^N}\begin{pmatrix}1\\ v\end{pmatrix} g_\infty\,dv = \begin{pmatrix}\rho_f\\ \rho_f u_f\end{pmatrix}.
\]
\item (Second moment)
\bq\label{eq:sm}
\int_{\mathbb R^N} (v-u_f)\otimes(v-u_f)\,g_\infty\,dv = \Big(\int_{\mathbb R^N} L_\psi(g_\infty)\,dv\Big)\,I_N.
\eq
\item (Explicit profile) \[
g_\infty(v)=\psi^{-1}\Big(\big(h_\psi(\rho_f)-\tfrac12|v-u_f|^2\big)_+\Big).
\]
\item (Constitutive relation)
\bq\label{eq:cr}
\rho_f h_\psi(\rho_f) = \int_{\mathbb R^N} L_\psi(g_\infty)\,dv + \int_{\mathbb R^N}\Big(\tfrac12|v-u_f|^2 g_\infty + \Psi(g_\infty)\Big)\,dv.
\eq
\item (Enthalpy relation)
\bq\label{eq:er}
h_\psi(\rho_f) = \frac{d}{d\rho_f}\int_{\mathbb R^N}\Big(\tfrac12|v-u_f|^2 g_\infty + \Psi(g_\infty)\Big)\,dv.
\eq
\end{enumerate}
\end{lemma}

\begin{proof}
The items are standard given \eqref{eq:ct} and \eqref{def:h}; we sketch the main steps. Mass and momentum conservation follow from integrating \eqref{eq:ct} against $1$ and $v$. Stationarity $\nabla_v(\psi(g_\infty)+V)=0$ on the support yields (3). Identity (4) follows from $L_\psi'(s)=s\psi'(s)$ and integrating by parts; differentiating with respect to $\rho_f$ gives (5).
\end{proof}

\begin{remark}\label{rmk:PbyL}
From the constitutive relation \eqref{eq:cr} and enthalpy relation \eqref{eq:er}, by recalling the definition of macroscopic pressure \eqref{eq:P,Phi}, we obtain
\[
P_\psi(\rho)= \int_{\R^N} L_\psi(g_\infty)\,\dv.
\]
\end{remark}

We now associate the long-time asymptotic $g_\infty$ with the local Maxwellian of $f$ for the collisional operator \eqref{eq:Q}.

\begin{definition}
For $f\in\mathcal A(\mathbb R^N)$, we define
\[
M_\psi[f]:=M_\psi^{(\rho_f,u_f)},\qquad 
M_\psi^{(\rho,u)}(v):=\psi^{-1}\Big(\big(h_\psi(\rho)-\tfrac12|v-u|^2\big)_+\Big).
\]
For $\psi_m$ as in Lemma \ref{prop:pl}, the local Maxwellian is given as
\[
M_{\psi_m}(v)= c_2 (c_1 \rho_f^{\gamma-1} - |v-u_f|^2)_+^{\frac{d}{2}},
\]
where $c_1$ and $c_2$ are defined as in \eqref{eq:c1c2}. This Maxwellian arises in the BGK model introduced by Bouchut \cite{bouchut1999}, see also \cite{CH24b, KS24} for more details. 
\end{definition}

From Lemma \ref{lem:a}, we obtain the following corollary.

\begin{corollary}\label{cor:basic}
For $f\in\mathcal A(\mathbb R^N)$ as defined in \eqref{eq:A}, the Maxwellian $M_\psi[f]$ satisfies:
\begin{enumerate}
\item (Minimization) 
\bq\label{eq:minpr}
\int_{\mathbb R^N}\Big(\tfrac12|v|^2 f + \Psi(f)\Big)\,dv 
\ge \int_{\mathbb R^N}\Big(\tfrac12|v|^2 M_\psi[f] + \Psi(M_\psi[f])\Big)\,dv.
\eq
\item (Compatibility) 
\bq\label{eq:comp}
\int_{\mathbb R^N}\Big(\tfrac12|v|^2 M_\psi[f] + \Psi(M_\psi[f])\Big)\,dv
= \tfrac12\rho_f|u_f|^2 + \Phi_\psi(\rho_f),
\eq
where $\Phi_\psi$ is defined in \eqref{eq:P,Phi}.
\item (Generalized Log-Sobolev inequality) 
\bq\label{eq:LSI}
\begin{aligned}
\int_{\mathbb R^N}\Big(\tfrac12|v|^2 f + \Psi(f)\Big)\,dv 
\,- &\int_{\mathbb R^N}\Big(\tfrac12|v|^2 M_\psi[f] + \Psi(M_\psi[f])\Big)\,dv \\
&\le \tfrac12\int_{\mathbb R^N} f\big|\nabla_v(\psi(f)+V)\big|^2\,dv.
\end{aligned}
\eq
\end{enumerate}
\end{corollary}

\begin{proof}
(1) follows from dissipation structure \eqref{eq:diss:g}. (2) is a direct consequence of \eqref{eq:cr} and \eqref{eq:P,Phi}. (3) is the generalized Log Sobolev inequality, see e.g. \cite[Theorem 17, Section 3.4]{CJMTU01}. See also Appendix \ref{app}. Concerning the generalized Log Sobolev inequality, we also refer to \cite{DD02} and the references therein.
\end{proof}

Now, multiplying $\begin{pmatrix} 1 \\v
\end{pmatrix}$ to \eqref{main_eq} and integrating against $\dv$, we obtain
\begin{align}\label{eq:f_cons}
\begin{aligned}
&\pa_t \rho_f + \nabla_x \cdot (\rho_f u_f) = 0,\\
&\pa_t (\rho_f u_f) + \nabla_x \cdot (\rho_f u_f \otimes u_f) + \nabla_x \cdot \lt(\int_{\mathbb R^N} (v-u_f)\otimes(v-u_f)\,f\dv\rt)=0.
\end{aligned}
\end{align}
From Remark \ref{rmk:PbyL} and \eqref{eq:sm},
\begin{align*}
\int_{\mathbb R^N} (v-u_f)\otimes(v-u_f)\,f\,dv - P_\psi(\rho_f)I_N
&= \int_{\mathbb R^N} (v-u_f)\otimes(v-u_f)\,(f-M_\psi[f])\,dv
\end{align*}
The rest of the analysis mainly concerns the control of error between the pressure tensor
$
\int_{\mathbb R^N} (v-u_f)\otimes(v-u_f)\,f\,dv
$
and barotropic pressure law $P_\psi(\rho_f)$. From the trivial inequality
\[
\Big|\int_{\mathbb R^N} (v-u_f)\otimes(v-u_f)\,(f-M_\psi[f])\,dv\Big|
\le \int_{\mathbb R^N} |v-u_f|^2\,|f-M_\psi[f]|\,dv,
\]
we aim to bound the right-hand side to obtain the passage from \eqref{main_eq} to the barotropic Euler equations with pressure $P_\psi$.

\section{Control of pressure error}\label{sec:key}

In this section, we prove the following quantitative estimate.

\begin{proposition}\label{prop:diss}
Assume (A1)--(A3) hold. For any admissible $f\in\mathcal A(\mathbb R^N)$, there exists $C=C_{\Lambda,N}>0$ such that
\begin{equation}\label{eq:key}
\int_{\mathbb R^N}\Big(\tfrac12|v-u_f|^2\,|f-M_\psi[f]| + \big|\Psi(f)-\Psi(M_\psi[f])\big|\Big)\,dv
\le C\Big(\sqrt{\Phi_\psi(\rho_f)}\sqrt{R\calF(f)} + R\calF(f)\Big),
\end{equation}
where
\bq\label{eq:D}
R\calF(f):= \calF(f) - \calF(M_\psi[f]) = \int_{\mathbb R^N}\Big(\tfrac12|v|^2 f + \Psi(f)\Big)\,dv 
- \int_{\mathbb R^N}\Big(\tfrac12|v|^2 M_\psi[f] + \Psi(M_\psi[f])\Big)\,dv.
\eq
\end{proposition}

\begin{remark}
Proposition \ref{prop:diss} is known in special cases for  $\psi_m$ in \cite[Lemma 1.4]{KS24} for $1<m<\infty$ and   for $\psi_\infty$ in \cite{BV05} (see Remark \ref{rmk:psi_m} for definitions of $\psi_m$). To treat the general degenerate diffusion case, we use the extension map technique inspired by \cite{KS24} and rearrangement ideas in \cite{BV05}. 
\end{remark}

\subsection{Extension of $f$ to dimension $N+1$}
Since  $\psi:\R \to \R$ is an invertible  $\mathcal C^1$ function, we have 
\[
s=\int_{\psi(0)}^{\psi(s)}(\psi^{-1})'(u)\,du.
\]
Moreover, it also holds that
\[
\int_0^s \psi(u)\,du = \int_{\psi(0)}^{\psi(s)} u (\psi^{-1})'(u)\,du.
\]
By changing variables $u=\tfrac12 I^2$, we obtain the following useful representation
\[
\tfrac12|v|^2 f + \Psi(f) = \int_0^{\sqrt{2\psi(f(v))}} \tfrac12\big(|v|^2+I^2\big)\,(\psi^{-1})'\!\big(\tfrac12 I^2\big)\,I\,dI.
\]
Inspired by \cite{KS24}, we define the \emph{extension map}
\[
\widehat f(v,I):= \mathbf 1_{\{0\le I\le \sqrt{2\psi(f(v))}\}},
\]
and the weight
\begin{equation}\label{eq:mu}
\mu(I):=(\psi^{-1})'\!\big(\tfrac12 I^2\big)\,I.
\end{equation}
Then, we have
\begin{equation}\label{eq:wh:recover}
f(v)=\int_0^\infty \widehat f(v,I)\,\mu(I)\,dI,\qquad
\tfrac12|v|^2 f + \Psi(f) = \int_0^\infty \tfrac12\big(|v|^2+I^2\big)\widehat f(v,I)\,\mu(I)\,dI.
\end{equation}
For the Maxwellian $M_\psi[f]$, its extension map is the characteristic function on a ball:
\begin{equation}\label{eq:wh:M}
\widehat{M_\psi[f]}(v,I) = \mathbf 1_{\{(v,I):\,|v-u_f|^2+I^2\le 2h_\psi(\rho_f)\}}.
\end{equation}
Translating $v\mapsto v+u_f$, we may assume $u_f=0$. It is crucial to notice that $\widehat{M_\psi[f]}$ is radial in $(v,I)$ in the half-space $\R^N \times \R_+$.

\subsection{Projection to one-dimensional space by radialization}

In this section, we radialize the extended map $\widehat{f}$. This type of projection idea appears in \cite{BV05}. Denoting the radius $r=\sqrt{|v|^2+I^2}$ for coordinates in $(v,I)\in\mathbb R^{N+1}_+$, we let
\[
\overline f(r):=\frac{1}{s(r)}\int_{\mathbb S^+_N}\widehat f(r\sigma)\,\mu(r\cos\theta)\,d\sigma,
\]
where $\theta$ denotes the angle between $(v,I)$ and $I$-axis in $(v,I)$-space, and $s(r)$ is defined as normalizing constant
\bq\label{eq:proj}
s(r):=\int_{\mathbb S^+_N}\mu(r\cos\theta)\,d\sigma(\theta).
\eq
Set 
\bq\label{eq:a}
a(r):=r^N s(r),\qquad dA(r):=a(r)\,dr.
\eq
Notice that $\overline{f}(r)$ is valued in $[0,1]$. Moreover, recalling \eqref{eq:wh:M}, the projected Maxwellian is given by $\overline M(r)=\mathbf 1_{[0,r_1]}(r)$ with $r_1:=\sqrt{2h_\psi(\rho_f)}$. 

\begin{lemma}\label{lem:recover}
For $f \in \mathcal A(\R^N)$, we have
\[
\rho_f=\int_0^\infty \overline f(r)\,d A(r),\qquad
\int_{\mathbb R^N}\Big(\tfrac12|v|^2 f + \Psi(f)\Big)\,dv
= \int_0^\infty \tfrac12 r^2 \overline f(r)\,d A(r).
\]
\end{lemma}
\begin{proof}
Notice that
\begin{align*}
\begin{aligned}
\rho_f &= \intr f(v)\,\dv =\intr \int_0^\infty \widehat{f}(v,I)\,\mu(I)\,dIdv \\
&= \int_0^\infty \int_{\bbS_N^+} \widehat{f}(r\sigma) \mu(r \cos \theta) r^N d\sigma(\theta) dr  = \int_0^\infty \overline{f}(r) r^N s(r)\,dr.
\end{aligned}
\end{align*}
The second equality can be proved similarly by \eqref{eq:wh:recover}.
\end{proof}

As a direct application of Lemma \ref{lem:recover}, we can rewrite  $R\calF(f)$ in \eqref{eq:D} as
\begin{equation}\label{aux}
 R\calF(f) = D=\int_0^\infty \tfrac12 r^2(\overline f(r)-\overline M(r))\,d A(r).
\end{equation}

\subsection{Decomposition of Dissipation}

The monotonicity of $\Psi$ and \eqref{eq:wh:recover} yields
\[
\begin{aligned}
& \int_{\mathbb R^N}\Big(\tfrac12|v|^2|f-M| + |\Psi(f)-\Psi(M)|\Big)\,dv  = \int_{\mathbb R^N}\Big|\tfrac12|v|^2(f-M) + (\Psi(f)-\Psi(M))\Big|\,dv  \\
= &\intr \lt|\int_0^\infty \tfrac12\big(|v|^2+I^2\big)(\widehat f(v,I) -\widehat{M_\psi[f]})\,\mu(I)\,dI\rt| \,dv\\
\le &\intr \int_0^\infty \tfrac12\big(|v|^2+I^2\big)\lt|\widehat f(v,I) -\widehat{M_\psi[f]}\rt|\,\mu(I)\,dI \,dv\\
= &\int_0^\infty \int_{\bbS_N^+}  \tfrac12r^2 \lt|\widehat f(r\sigma) -\widehat{M_\psi[f]}(r\sigma)\rt|\,\mu(r\cos\theta)\,d\sigma(\theta)r^Ndr \\
= &\int_0^{r_1} \int_{\bbS_N^+}  \tfrac12r^2 (1-\widehat f(r\sigma)) \,\mu(r\cos\theta)\,d\sigma(\theta)r^Ndr + \int_{r_1}^{\infty} \int_{\bbS_N^+}  \tfrac12r^2 \widehat f(r\sigma) \,\mu(r\cos\theta)\,d\sigma(\theta)r^Ndr \\
=  &\int_0^{r_1} \tfrac12 r^2\,(1-\overline f(r))\,dA(r) + \int_{r_1}^\infty  \tfrac12 r^2\,\overline f(r)\,dA(r)
= \int_0^{\infty} \tfrac12 r^2\,|\overline f(r)-\overline M(r)|\,dA(r) =:F.
\end{aligned}
\]
Note that the fifth identity holds by separating the integral into $r \le r_1$ and $r>r_1$ parts since $\widehat{M_\psi[f]}(r\sigma)=1_{r\le r_1}$ and $\widehat{f}(r\sigma) \in \{0,1\}$. Similarly, the last identity holds because $\ov{f}(r) \in [0,1]$. It is now easy to see that $F\geq D$ by \eqref{aux}. Estimating $F$ in terms of $D$ is challenging due to cancellations in the integrand of $D$.
Following \cite{BV05}, we decompose $F$ and $D$ into local and global parts and treat each case in a suitable way.

We first notice that Lemma \ref{lem:recover} implies
\[
\int_0^{r_1}(1-\ov{f}) dA =\rho_f - \int_0^{r_1} \overline{f} dA = \int_{r_1}^\infty \overline{f} dA.
\]
In other words, the mass fluctuations are the same for the $r\le r_1$ and $r >r_1$ parts. Now, we may assume $\ov{f}\neq \ov{M}$, otherwise $F=D=0$. We decompose the fluctuations near $r_1$ and away from $r_1$. Namely, we
fix $r_0\in(0,r_1)$ and uniquely choose $r_2>r_1$ in the following way:
\[
r_2:=\min \lt\{\tilde r >r_1: \int_{r_1}^{\tilde r} \ov{f}(r)\,dA(r) = \int_{r_0}^{r_1} (1-\ov{f}(r))\,dA(r)\rt\}.
\]
We then obtain
\bq\label{eq:r1r2}
M_{\rm loc}:=\int_{r_0}^{r_1}(1-\overline f)\,dA = \int_{r_1}^{r_2}\overline f\,d A,\qquad
M_{\rm glo}:=\int_0^{r_0}(1-\overline f)\,d A = \int_{r_2}^\infty \overline f\,d A.
\eq
Notice that 
\[
M_{\rm loc} + M_{\rm glo} = \int_0^{r_1} (1-\overline{f}) dA = \int_{r_1}^\infty \overline{f} dA.
\]
We similarly define $F_{\rm loc},F_{\rm glo},D_{\rm loc},D_{\rm glo}$ in the following way.
\[
\begin{aligned}
&F_{\rm loc}:= \int_{r_0}^{r_1} \frac{1}{2}r^2(1-\overline{f}(r))\,dA(r) + \int_{r_1}^{r_2} \frac{1}{2} r^2 \overline{f}(r)\,dA(r)  \\
&F_{\rm glo}:= \int_{0}^{r_0} \frac{1}{2}r^2(1-\overline{f}(r))\,dA(r) + \int_{r_2}^{\infty} \frac{1}{2} r^2 \overline{f}(r)\,dA(r)\\
&D_{\rm loc}:= \int_{r_0}^{r_1} \frac{1}{2}r^2(\overline{f}(r)-1)\,dA(r) + \int_{r_1}^{r_2} \frac{1}{2} r^2 \overline{f}(r)\,dA(r) \\
&D_{\rm glo}:= \int_{0}^{r_0} \frac{1}{2}r^2(\overline{f}(r)-1)\,dA(r) + \int_{r_2}^{\infty} \frac{1}{2} r^2 \overline{f}(r)\,dA(r)
\end{aligned}
\]
Using monotonicity of $r\mapsto r^2$ one obtains the elementary bound
\[
F_{\rm glo} - D_{\rm glo} = \int_0^{r_0} r^2 (1- \overline{f}(r))\,dA(r) \le r_0^2 M_{\rm glo}.
\]
On the other hand, from the definition of $D_{\rm glo}$ we deduce that
\[
\begin{aligned}
D_{\rm glo} &\ge  r_0^2 \int_{0}^{r_0} \frac{1}{2} (\overline{f}(r)-1)\,dA(r)  + r_2^2 \int_{r_2}^\infty \frac{1}{2} \overline{f}(r)\,dA(r) \\
&= \frac12 (r_2^2-r_0^2)M_{\rm glo}.
\end{aligned}
\]
Combining these two, we get
\bq\label{eq:glo}
F_{\rm glo}\le\Big(\frac{r_2^2+r_0^2}{r_2^2-r_0^2}\Big)D_{\rm glo}.
\eq

\subsection{Control of the local (critical) part}

To control the $F_{\rm loc}$ term, we make use of the following observation: namely, we can choose $\alpha>r_0$ and $r_2>\beta>0$ such that
\begin{align}\label{eq:albe}
M_{\rm loc} = \int_{r_0}^{r_1}  (1 - \ov{f}(r))\,\dA(r) = \int_{\alpha}^{r_1}  \,\dA(r) = \int_{r_1}^{r_2}  \ov{f}(r)\,\dA(r) = \int_{r_1}^{\beta}  \,\dA(r)\,,
\end{align}
since $0\leq 1 - \ov{f}(r)\leq 1$ and $0\leq \ov{f}(r)\leq 1$.
We then note that
\bq\label{eq:alphabeta}
D_{\rm loc} \ge  - \int_{\alpha}^{r_1} \frac{1}{2}r^2 \,\dA(r) + \int_{r_1}^{\beta} \frac{1}{2}r^2 \,\dA(r)=:D_{\rm cri}.
\eq

Note that the right-hand side $D_{\rm cri}$ can be written as
\bq\label{cri}
D_{\rm cri} = \frac{1}{2}\lt(A_2(\beta) - 2A_2(r_1) +A_2(\alpha)\rt)
\eq
where $A_2(x):=\int_0^x r^2 dA(r)$.
This motivates the following auxiliary result, which is the main technical lemma of the paper.

\begin{lemma}\label{lem:crit}
Let $b:[0,\infty)\to [0,\infty)$ satisfy the doubling condition
\begin{equation}\label{eq:adou}
\frac{b(r_1)}{\Lambda_*}  \le b(r_2) \le \Lambda_* b(r_1)
\end{equation}
for some $\Lambda_* \ge 1$ whenever $0 \le r_1 \le r_2 \le 2r_1$. We define 
\[
B_0(R):=\int_0^{R} b(r)\,\dr, \quad B_2(R):=\int_0^R r^2b(r)\,\dr,
\]
for any $R \ge 0$. Suppose that
\bq\label{eq:A_0:bal}
B_0(x) - B_0(y) = B_0(y) - B_0(z)
\eq
for some $0< \frac{1}{2}y \le z< y< x \le 2y$. Then we have
\[
\lt( B_2(x) - B_2(z)\rt)^2\le 128 \Lambda_*^4 y^2 B_0(y)\lt(B_2(x) - 2B_2(y) +B_2(z)\rt).
\]
\end{lemma}

\begin{proof}
We first decompose
\begin{align*}
B_2(x) - B_2(y) = \int_y^x r^2 b(r)\,\dr 
= y^2 \lt[B_0(x) - B_0(y)\rt] +  \int_y^x (r^2 - y^2)b(r)\,\dr,
\end{align*}
and obtain a lower bound via doubling condition
\begin{align*}
\int_y^x (r^2 - y^2)b(r)\,\dr &\ge \frac{b(y)}{\Lambda_*} \int_y^x  (r^2-y^2)\,\dr.
\end{align*}
Similarly, we decompose
\begin{align*}
B_2(z) - B_2(y) &= -\int_z^y r^2 b(r)\,\dr = -y^2 [B_0(y)- B_0(z)] + \int_z^y (y^2-r^2)b(r)\,\dr
\end{align*}
and derive a lower bound estimate
\begin{align*}
 \int_z^y (y^2-r^2)b(r)\,\dr \ge \frac{b(y)}{\Lambda_*}\int_z^y (y^2-r^2)\,\dr
\end{align*}
where the last inequality follows from the doubling condition.
Summing these two and applying \eqref{eq:A_0:bal}, we obtain
\begin{align}\label{pasta}
B_2(x)-2B_2(y) + B_2(z) &\ge \frac{b(y)}{ \Lambda_*} \int_z^x |y^2-r^2|\,\dr \ge \frac{yb(y)}{\Lambda_*} \int_z^x |y-r|\,\dr \ge \frac{yb(y)}{4\Lambda_*}(x-z)^2,
\end{align}
where we have used $z<y<x$ for the last inequality and minimizing on $y$.
On the other hand,
\[
B_2(x) - B_2(y) \le \Lambda_*b(y)\int_y^x r^2\,\dr 
\]
where we have employed $x < 2y$ and the doubling condition. Similarly, we also get
\[
B_2(y) -B_2(z) \le  \Lambda_* b(y)\int_z^y r^2\,\dr.
\]
Combining these two,
\bq\label{pizza}
B_2(x) -B_2(z) \le  \Lambda_* b(y) \int_z^x r^2\,\dr \le 4\Lambda_* y^2 b(y) (x-z),
\eq
where we have applied $x \le 2y$ for the last inequality.
Hence, by combining \eqref{pasta} and \eqref{pizza}, we obtain
\begin{align}\label{carbonara}
[B_2(x) -B_2(z)]^2 &\le 64 \Lambda_*^3y^3 b(y) \lt(B_2(x) - 2B_2(y) +B_2(z) \rt).
\end{align}
Finally, we notice that
\[
B_0(y) = \int_0^y b(r)\,\dr \ge \int_{y/2}^y b(r) \,\dr \ge \int_{y/2}^y\frac{b(y)}{\Lambda_*} \,\dr = \frac{yb(y)}{2\Lambda_*}.
\]
Inserting this inequality into \eqref{carbonara}, we conclude the proof.
\end{proof}

To apply Lemma \ref{lem:crit} to $a(r)$ defined in \eqref{eq:a}, we need to confirm the following:
\begin{lemma}\label{lem:a_conf}
The function $a(\cdot)$ defined in \eqref{eq:a} satisfies the doubling condition \eqref{eq:adou}.
\end{lemma}

\begin{proof}
We recall \eqref{eq:a} that $a(r)= r^N s(r)$. Thus, it suffices to study the doubling property of $s(r)$. Indeed, from \eqref{eq:mu} and \eqref{eq:proj}, we have
\[
s(r)= \int_{\bbS_{N}^+} \mu(r\cos\theta)\dsi,\quad \mu(I)= (\psi^{-1})'\lt(\frac{1}{2}I^2\rt) I. 
\]
For $r_2 \le 2r_1$, by (A3)  \eqref{eq:double}, we notice that
\[
s(r_2) \le 2\Lambda^2 s(r_1).
\]
Similarly, $r_1 \le r_2$ implies
\[
s(r_2) \ge \frac{1}{\Lambda^2}s(r_1).
\]
Hence, we obtain
\[
\frac{1}{\Lambda^2}a(r_1) \le a(r_2) \le 2^{N+1}\Lambda^2 a(r_1),
\]
satisfying \eqref{eq:adou} with $\Lambda_*:= 2^{N+1} \Lambda^2$.
\end{proof}

\begin{proof}[Proof of Proposition \ref{prop:diss}]
The delicate part of the proof lies in the estimate of $F_{\rm loc}$ term, where we apply Lemma \ref{lem:crit} to $a(r)$ as defined in \eqref{eq:a}. To this end, we first choose $r_0:=\lt(1-\tfrac{1}{2\Lambda_*^2}\rt)r_1$ where $\Lambda_* \ge 1$ as in the proof of Lemma \ref{lem:a_conf}. Recall the definitions of $r_2$ in \eqref{eq:r1r2}. Since $r_2 \ge r_1$, we have
\[
\frac{r_2^2+r_0^2}{r_2^2-r_0^2} = 1 + \frac{2}{\frac{r_2^2}{r_0^2}-1}\le 1+\frac{2}{\lt(1-\frac{1}{2\Lambda_*^2}\rt)^{-2}-1}=:C^{(1)},
\]
so that
\bq\label{eq:Fglo}
F_{\rm glo} \le C^{(1)}D_{\rm glo}
\eq
by \eqref{eq:glo}, where $C^{(1)}$ depends on $\Lambda$ and $N$.

We then recall the definition of $\alpha$, $\beta$ in \eqref{eq:albe}. It is easy to notice that $\frac{1}{2}r_1\le r_0\le \alpha\le r_1$. Next, we claim $\beta \le 2r_1$. Suppose not, then \eqref{eq:albe} and Lemma \ref{lem:a_conf} implies
\[
M_{\rm loc} = \int_{r_1}^{\beta} dA > \int_{r_1}^{2r_1}dA \ge \frac{a(r_1)}{\Lambda_*}r_1.
\]
On the other hand, using again the doubling condition on $a$ due to Lemma \ref{lem:a_conf} 
\[
M_{\rm loc} = \int_{\alpha}^{r_1} dA \le \Lambda_*a(r_1)(r_1-\alpha) \le \Lambda_*a(r_1)(r_1-r_0)=\frac{a(r_1)}{2\Lambda_*}r_1,
\]
yielding a contradiction. This proves the claim.

By applying Lemma \ref{lem:crit} to $a(r)$ with $(x,y,z):=(\alpha, r_1, \beta)$, we deduce that there exists a constant $C_{N,\Lambda}$ such that
\[
F_{\rm loc}^2 \le C_{N,\Lambda}^2\,r_1^2 \rho_f\,D_{\rm cri}
\]
 using \eqref{cri} and that $B_0(r_1)=A(r_1) = \rho_f$. Appealing \eqref{eq:alphabeta} to recall $D_{\rm cri} \le D_{\rm loc}$ and that $r_1^2=2h_\psi(\rho_f)$ by definition, we obtain
$$
F_{\rm loc}\le C_{N,\Lambda}\,\sqrt{2\rho_f h_\psi(\rho_f)}\,\sqrt{D_{\rm loc}}.
$$
Finally, we recall Corollary \ref{cor:sh(s)} that shows
$$
\rho_f h_\psi(\rho_f)\le 2^{\frac{2}{N}+1}\Phi_\psi(\rho_f),
$$
that finally implies
\bq\label{eq:Floc}
F_{\rm loc}\le C_{N,\Lambda}\,\sqrt{\Phi_\psi(\rho_f)}\,\sqrt{D_{\rm loc}},
\eq
with a suitable constant $C_{N,\Lambda}$ that we denote the same for the sake of simplicity.

Since $D_{\rm glo}, D_{\rm loc} \ge 0$, combining the estimates \eqref{eq:Fglo}--\eqref{eq:Floc} for the local and global parts gives the desired bound \eqref{eq:key}. 

\end{proof}

\section{Main Results: Hydrodynamic limit}\label{sec:mr}

\subsection{Statement of main results}

Our main results concern the convergence of macroscopic quantities associated to \eqref{main_eq} toward the following barotropic Euler equations as $\e \to 0$:
\bq \label{E}
\begin{cases}
\pa_t \rho + \nabla_x \cdot (\rho u) = 0,\\
\pa_t (\rho u ) + \nabla_x \cdot (\rho u \otimes u) + \nabla_x P_\psi(\rho) = 0.
\end{cases}
\eq

The aim is to approximate the solution of \eqref{E} from the local conservation law \eqref{eq:f_cons}. To this end, we briefly introduce the notion of relative entropy. This idea is based on the weak-strong uniqueness principle \cite{Daf79}. The entropy of \eqref{E} is defined as 
\[
\eta(U):= \frac{1}{2}\rho |u|^2 + \Phi_\psi(\rho)
\]
where $U:=(\rho,\rho u)^T$ and $\Phi_\psi$ is defined in \eqref{eq:P,Phi}. Recall that $\Phi_\psi$ is convex (Lemma \ref{cor:sh(s)}). For classical solutions of \eqref{E}, we have 
\[%\label{eq:conE}
\frac{d}{dt}\int\frac{1}{2}\rho |u|^2 + \Phi_\psi(\rho) \, \dx = 0.
\]
The relative entropy $\eta(U^\e|U)$ is defined as 
\[%\label{eq:re}
\begin{split}
\eta(U^\e|U) :=& \,\eta(U^\e) - \eta(U) - d\eta(U) \cdot (U^\e-U) \cr
=&  \,\frac{\rho^\e}{2}|u^\e-u|^2 + \Phi_\psi(\rho^\e|\rho)
\end{split}
\]
where
\[
d\eta(U)=(h_\psi(\rho)-\frac{1}{2}|u|^2,u)^T
\]
and
\[
\Phi_\psi(\rho^\e|\rho) = \Phi_\psi(\rho^\e) - \Phi_\psi(\rho) - \Phi_\psi'(\rho)(\rho^\e-\rho).
\]

We now state our main result on the hydrodynamic limit to barotropic Euler equations \eqref{E} from the kinetic Fokker--Planck equations \eqref{main_eq}.
\begin{theorem}\label{thm:hdr}
Assume (A1)--(A5) hold. Let $(\rho, u) \in \calC^1([0,T^*]\times\R^N)$ be given a classical solution of \eqref{E} having finite initial entropy $\eta(U_0) \in L^1(\R^N)$. For any $\e \in(0,1]$, let $f^\e$ be a weak solution of \eqref{main_eq} satisfying the kinetic entropy inequality 
\bq\label{eq:kep}
\iint H(f^\e(t),v)\,\dx\dv + \frac{1}{\e}\int_0^t \intr  I(f^\e(s))\,\dx\ds \le \iint H(f^\e_0,v)\,\dx\dv
\eq
for all $0\le t \le T^*$, where  
\[
H(f,v):= \frac{1}{2}|v|^2 f + \Psi (f)
\]
with the dissipation term $I(\cdot)$ as defined in \eqref{eq:F,I}.

We further assume that initial datum $f^\e_0$ is well-prepared in the sense that $\|f^\e_0\|_{L^1}=\|\rho_0\|_{L^1}=1$,
\[
\textnormal{(H0)}\quad \int \eta(U_0^\e)\,\dx \le C_A,\qquad \textnormal{(H1)}\quad \int \eta(U^\e_0|U_0)\,\dx  \le C_A\e^{1/2},
\]
together with compatibility condition
\[
\textnormal{(H2)}\quad \iint H(f^\e_0,v)\,\dx\dv \le  \int \eta(U^\e_0)\dx + C_A\e^{1/2},
\]
then there exists $C_*>0$ depending only on $C_A, N, T^*, \psi$, $\|\nabla u\|_{L^\infty((0,T^*)\times\R^N)}$, and $\|\eta(U_0)\|_{L^1}$  such that the following quantitative estimates hold for the relative entropy:
\bq\label{mnre}
\sup_{0\le t \le T^*}\int \eta(U^\e_t|U_t)\,\dx \le C_*\e^{\frac{1}{2}}.
\eq
As a result, we have strong convergence of $\rho^\e \to \rho$, $\rho^\e u^\e \to \rho u$ and $\rho^\e u^\e \otimes u^\e \to \rho u \otimes u$ in $L^\infty([0,T^*];L^1_x)$ and $\intr (v-u^\e)\otimes(v-u^\e)f^\e\,\dv \to P_{\psi}(\rho) $ in $L^1_{\rm loc}([0,T^*]\times \R^N)$.
\end{theorem}

\begin{remark}
We remark that the condition (H2) is related to the following inequality
\[
\int H(f_0^\e,v)\,\dv \ge \eta(U_0^\e)
\]
 by Corollary \ref{cor:basic}. The (H2) condition ensures that the kinetic entropy behaves as its macroscopic part.
\end{remark}

\begin{corollary}\label{quant}
Under the hypotheses of Theorem \ref{thm:hdr} and in the particular case of  $\psi=\psi_m$ as in Lemma \ref{prop:pl} so that $P_{\psi_m}(\rho)=\theta \rho^\gamma$, we derive quantified convergence rate of macroscopic quantities as: 
\begin{itemize}
\item the density
\[
\|\rho^\e - \rho\|_{L^\infty(0,T^*;L^\gamma)} \le C_*\e^{\beta},
\]
\item
the momentum
\[
\|\rho^\e u^\e - \rho u\|_{L^\infty(0,T^*;L^{\frac{2\gamma}{\gamma+1}})} \le C_*\e^{\beta},
\]
\item
the flux
\[
\|\rho^\e u^\e\otimes u^\e - \rho u\otimes u\|_{L^\infty(0,T^*;L^1)} \le C_*\e^{\beta}, 
\]
\item
the pressure tensor
\[
\lt\|\int (v-u^\e)\otimes (v-u^\e) f^\e \dv -\theta\rho^\gamma \mathbb{I}_N \rt\|_{L^1(0,T^*;L^1)} \le C_*\e^{\beta},
\]
\end{itemize}
where $\beta=\frac{1}{2\max\{2,\gamma\}}$.
\end{corollary}

\subsection{Estimates for relative entropy}
In this section, we prove \eqref{mnre} in Theorem \ref{thm:hdr}. We denote
\[
U:=(\rho, \rho u)^T,\quad U^\e:= (\rho_{f^\e}, m_{f^\e})^T
\]
where $(\rho, u)$ is a local strong solution of \eqref{E} and $f$ is a weak solution of \eqref{main_eq} satisfying the assumptions stated in Theorem \ref{thm:hdr}. The proof is rather standard; we refer to \cite{BV05} for the proof. 

\begin{lemma}\label{lem:fml} It holds that
\[
\begin{split}
\frac{d}{dt}\int \eta(U^\e_t|U_t) \dx =  \frac{d}{dt} \int \eta(U^\e_t)\dx\, -\int \nabla u : \lt(G(U^\e_t|U_t) + \calR^\e \rt) \dx 
\end{split}
\]
where 
\bq\label{eq:G}
G(U^\e_t|U_t):= \rho^\e (u^\e -u)\otimes (u^\e - u) + P_\psi (\rho^\e|\rho)\bbI_N
\eq
and
\bq\label{eq:pe}
\calR^\e:= \int_{\mathbb R^N} (v-u^\e)\otimes(v-u^\e)\,f^\e\,dv - P_\psi(\rho^\e)\bbI_N
\eq
\end{lemma}

\begin{lemma}\label{olive}
Under (A5), we have the following pointwise estimate
\[
|P_\psi(\rho^\e|\rho)| \le C \Phi_\psi(\rho^\e|\rho)
\]
for any $\rho^\e,\rho \ge 0$.
\end{lemma}
\begin{proof}
Notice that
\[
P_\psi''(\rho)= (\rho h_\psi'(\rho))' = h_\psi'(\rho) +\rho h_\psi''(\rho),\quad \Phi_\psi''(\rho)=h_\psi'(\rho).
\]
In particular, (A5) implies that
\bq\label{jamon}
\rho |P_\psi''(\rho)| \le C h_\psi'(\rho).
\eq
By Taylor's expansion (as \eqref{ring} in the sequel), we find that
\[
\Phi_\psi(\rho^\e|\rho)= \lt[\int_0^1 (1-s) h_\psi'((1-s)\rho + s\rho^\e)\,\ds\rt](\rho^\e -\rho)^2.
\]
Similarly,
\[
P_\psi(\rho^\e|\rho)= \lt[\int_0^1 (1-s) P_\psi''((1-s)\rho + s\rho^\e) \,\ds\rt](\rho^\e -\rho)^2.
\]
Since $h_\psi'(\rho) \ge 0$ and $|P_\psi''(\rho)| \le h_\psi'(\rho) + \rho|h_\psi''(\rho)| \le Ch_\psi'(\rho)$ by appealing to \eqref{jamon}, we obtain $|P_\psi(\rho^\e|\rho)| \le C \Phi_\psi(\rho^\e|\rho)$ by triangle's inequality.
\end{proof}

\begin{lemma}\label{lem:eps}
Under the hypothesis {(H0)--(H2)} in Theorem \ref{thm:hdr}, we have
\[
\int \eta(U^\e_t|U_t) \,\dx \le 
C \e^{\frac{1}{2}}, \quad \forall t \in [0,T^*]
\]
where $C>0$ is independent of $\e>0$.
\end{lemma}
%depends on $\gamma,\theta, n, T^*$ and $ \|\nabla u \|_{L^\infty(0,T^*;L^\infty)}$

%By integrating the equation \eqref{FPVP}  against $\begin{pmatrix} 1 \\ v \end{pmatrix}$

\begin{proof}
By Lemma \ref{lem:fml}, we have
\[
\begin{split}
\intr \eta(U^\e_t|U_t) \,\dx &=  \intr \eta(U^\e_0|U_0) \,\dx +\lt(\intr \eta(U^\e_t) \,\dx - \intr \eta(U^\e_0) \,\dx \rt) \cr
&\quad - \int_0^t \int \nabla u : G(U^\e_s|U_s) \,\dx\ds
- \int_0^t\int \nabla u : \calR^\e \dx \ds \cr
&=: \sum_{i=1}^4 J_i
\end{split}
\]
for any $t\in[0,T^*]$. First, it is clear from {(H1)} that 
\[
J_1 \le C_A\e^{1/2}.
\]
To estimate $J_2$, we proceed as
\[%\label{eq:etea:bdd}
\begin{split}
\int\eta(U^\e_t)\,\dx &= \iint H(M_\psi(f^\e)(t),v) \dx\dv \quad (\because \eqref{eq:comp}) \\
&\le  \iint H(f^\e(t),v) \dx\dv \quad (\because \eqref{eq:minpr}) \cr
&\le  \iint H(f^\e_0,v) \dx\dv  \quad(\because \eqref{eq:kep})\\
&\le \int \eta(U^\e_0)\,\dx + C_A\e^{1/2} \quad(\because (H2)),
\end{split}
\]
hence,
\[
J_2 \le C_A\e^{1/2}.
\]
Since $|G(U^\e|U)|_1 \le C \eta(U^\e|U)$ where
$
|A|_1:=\sum_{i,j}|a_{ij}|
$
for $A \in M(\R^{N\times N})$ by recalling \eqref{eq:G} and Lemma \ref{olive}, we obtain
\[
J_3 \le C\|\nabla u\|_{L^\infty([0,T^*]\times \R^N)} \int_0^t \int \eta(U^\e_s|U_s)\,\dx\ds.
\]
Lastly, for the estimate of $J_4$, we apply Proposition \ref{prop:diss} to find that
\[
\begin{split}
J_4 \le &C\|\nabla u\|_{L^\infty([0,T^*]\times \R^N)} \int_0^t \intr \Big(\sqrt{\Phi_\psi(\rho_{f^\e})}\sqrt{R\calF(f^\e)} + R\calF(f^\e)\Big)\,\dx\ds\\
\le &C\|\nabla u\|_{L^\infty([0,T^*]\times \R^N)} \int_0^t \lt(\intr {\Phi_\psi(\rho_{f^\e})}\,\dx\rt)^{\frac{1}{2}}\lt(\int_{\R^N}{R\calF(f^\e)}\dx\rt)^{\frac{1}{2}} + \lt(\int_{\R^N} R\calF(f^\e)\,\dx\rt)\ds\\
\le &C\|\nabla u\|_{L^\infty([0,T^*]\times \R^N)} \left[\sqrt{t}\sup_{s\in[0,t]}\lt(\intr {\Phi_\psi(\rho_{f^\e})}\,\dx\rt)^{\frac{1}{2}} \lt(\int_0^t \int_{\R^N}{R\calF(f^\e)}\dx\ds\rt)^{\frac{1}{2}} \right.\\
&\qquad\qquad\qquad\qquad\qquad \qquad \left.+ \lt(\int_0^t\int_{\R^N} R\calF(f^\e)\,\dx\ds\rt)\right]
\end{split}
\]
Notice that $\intr \Phi_\psi(\rho_{f^\e})\,\dx$ is uniformly bounded as
\bq\label{eq:Pub}
\begin{aligned}
\intr \Phi_\psi(\rho_{f^\e})\,\dx &\le \iint H(f^\e,v)\,\dx\dv\quad  (\because \eqref{eq:minpr}, \eqref{eq:comp})\\
&\le \iint H(f^\e_0,v) \,\dx\dv \quad (\because \eqref{eq:kep}) \\
&\le \intr \eta(U^\e_0) \dx + C_A\e^{1/2} \quad (\because (H2))\\
&\le C \quad(\because (H0)).
\end{aligned}
\eq
On the other hand, by the generalized Log-Sobolev inequality \eqref{eq:LSI} and the kinetic entropy inequality \eqref{eq:kep}, we obtain
\[
\int_0^t\intr R\calF(f^\e)\,\dx\ds  \le \int_0^t \intr I(f^\e)\,\dx\ds \le \e \iint H(f^\e_0,v)\,\dx\dv \le C\e.
\]
This proves $J_4 \le C(\e^{\frac{1}{2}}+\e)$. Combining all, we obtain
\[
\int \eta(U^\e_t|U_t) \dx \le C\e^{\frac{1}{2}} + C \int_0^t \int \eta(U^\e_s|U_s) \dx \ds, \quad \forall t\in[0,T^*]
\]
for all small $\e>0$. By Gr\"onwall's inequality, the conclusion of the lemma follows.
\end{proof}

\subsection{Convergence of macroscopic quantities: general case}\label{ssc:strcv}
Based on the estimates on the relative entropy \eqref{mnre}, we aim deduce the convergence of macroscopic variables associated to $f^\e$. 

Fix $\delta>0$ small enough (such that $3\delta<1/{\delta}$) and a bounded open set $\Omega \subset \R^N$. Since $h_\psi$ is strictly increasing $\calC^1$ function on $(0,\infty)$ thanks to Lemma \ref{lem:h}, we have $\inf_{s \in [\delta,1/{\delta}]} h_\psi'(s) =: c_\delta>0$. This yields, by means of mean value theorem,
\[
\Phi_\psi(\rho^\e|\rho)1_{\{\delta \le \rho^\e, \rho \le 1/{\delta}\}} \ge \frac{1}{2}c_{\delta}(\rho^\e -\rho)^21_{\{\delta \le \rho^\e, \rho \le 1/{\delta}\}}.
\]
This provides
\[
\||\rho^\e -\rho|1_{\{\delta \le \rho^\e, \rho \le 1/{\delta}\}}\|_{L^1(\Omega)}\le \lt(\frac{2|\Omega|}{c_\delta}\rt)^{\frac{1}{2}}\lt(\int \Phi_\psi(\rho^\e|\rho)\,\dx\rt)^{\frac{1}{2}}\le \lt(\frac{2|\Omega|C_*}{c_\delta}\rt)^{\frac{1}{2}}\e^{\frac{1}{4}}
\]
by H\"older's inequality and \eqref{mnre}. On the other hand, we recall \eqref{eq:Pub} that $\Phi_\psi(\rho^\e)$ are uniformly bounded in $L^\infty([0,T^*];L^1(\R^N))$, and 
\[
\Phi_\psi(s) = \int_0^sh_\psi(u)\,\du \ge \frac{s}{2}h_\psi\lt(\frac{s}{2}\rt)
\]
so that $\frac{\Phi_\psi(s)}{s} \to \infty$ as $s \to \infty$. This implies that $\rho^\e$ are uniformly integrable;  there exists  $\mu(\delta)\to 0$ as $\delta \to 0$
such that \[
\int_{\omega} \rho^\e +\rho\,\dx \le \mu(\delta)
\]
whenever a measurable set $\omega\subset \Omega$ satisfies $|\omega| = O(\delta)$. On the other hand, we also have
\[
\lt|\{\rho^\e \ge 1/{\delta}\} \rt| \le \delta \int \rho^\e\,\dx \le C\delta.
\]
Analogously, we also have $\lt|\{\rho \ge 1/{\delta}\} \rt| \le C\delta.$
Notice that these estimates hold uniformly in $\e  \in (0,\frac{1}{2}]$, $t \in [0,T^*]$. We then decompose $\Omega$ into four subsets:
\[
\begin{aligned}
&{I}:=\{x\in\Omega: \rho^\e,\rho \le \delta\},\quad II:=\{x\in\Omega: \rho^\e \le \delta < \rho \le {1}/{\delta} \} \cup \{x\in\Omega: \rho \le \delta <\rho^\e \le 1/{\delta} \}\\
&III:= \{x\in\Omega: \delta< \rho^\e,\rho \le 1/{\delta} \},\quad IV:=\{x\in\Omega: \rho^\e > {1}/{\delta} \} \cup \{x\in\Omega: \rho >1/{\delta} \}.
\end{aligned}
\]
It is obvious that
\[
\|\rho^\e-\rho\|_{L^1(I)}\le \delta|I|,\quad \|\rho^\e-\rho\|_{L^1(III)}\le \lt(\frac{2|\Omega|C_*}{c_\delta}\rt)^{\frac{1}{2}}\e^{\frac{1}{4}},\quad \|\rho^\e-\rho\|_{L^1(IV)}\le 2\mu(\delta).
\]
To bound the integral of $\rho^\e -\rho$ on $II$, we further decompose this set into two parts:
\[
\begin{aligned}
&II_{\rm good}:=\{x\in\Omega: \rho^\e \le \delta \le  \rho\le 5\delta \} \cup \{x\in\Omega: \rho \le \delta <\rho^\e \le 5\delta \},\\
&II_{\rm bad}:=II_{\rm bad}^{(1)}\cup II_{\rm bad}^{(2)}=\{x\in\Omega: \rho^\e \le \delta \le   5\delta < \rho \le 1/{\delta} \} \cup \{x\in\Omega: \rho \le \delta  \le 5\delta < \rho^\e \le 1/{\delta} \}.
\end{aligned}
\]
The estimate on $II_{\rm good}$ is immediate:
\[
\|\rho^\e -\rho\|_{L^1(II_{\rm good})}\le 5\delta |II_{\rm good}|.
\]
We then estimate the integral on $II_{\rm bad}^{(2)}$. Let us define $l_x:[0,1]\to \R_+$ by setting
\[
l_x(s):=\rho(x) + s(\rho^\e(x)-\rho(x))
\]
whenever $x \in II_{\rm bad}^{(2)}$. We set $g(s):=(\Phi_\psi \circ l_x)(s)=\Phi_\psi(l_x(s))$ for $s\in[0,1]$. Notice that
\bq\label{ring}
g(1)-g(0)-g'(0)=\int_0^1(1-s)g''(s)\,\ds
\eq
with $g'(s)=h_\psi(l_x(s))(\rho^\e(x)-\rho(x))$, $g''(s)=h_\psi'(l_x(s))(\rho^\e(x)-\rho(x))^2$. Since
\[
1/{\delta}\ge \rho^\e(x) \ge l_x(s) \ge \frac{1}{4}(\rho^\e(x)-\rho(x))\ge {\delta}\quad \text{for all}\,\,\,s\in \lt[\frac{1}{4},\frac{1}{2}\rt],
\]
we obtain a pointwise estimate: for all $x\in II_{\rm bad}^{(2)}$,
\[
\begin{aligned}
\Phi_\psi(\rho^\e(x)|\rho(x)) = \lt[\int_0^1 (1-s)h_\psi'(l_x(s))\,\ds \rt](\rho^\e(x)-\rho(x))^2 &\ge \frac{1}{2}\lt[\int_{\frac{1}{4}}^{\frac{1}{2}} h_\psi'(l_x(s))\,\ds \rt](\rho^\e(x)-\rho(x))^2 \\
&\ge \frac{c_\delta}{8}(\rho^\e(x)-\rho(x))^2.
\end{aligned}
\]
The $II_{\rm bad}^{(1)}$ case can be estimated similarly. Combining all, we find that
\[
\|\rho^\e -\rho\|_{L^\infty([0,T^*];L^1(\Omega))}\le 5\delta |\Omega| + \lt(\frac{8|\Omega|C_*}{c_\delta}\rt)^{\frac{1}{2}}\e^{\frac{1}{4}} + 2\mu(\delta).
\]
By taking limsup $\e \to 0$ first and choosing $\delta>0$ arbitrarily small, we deduce that
\bq\label{crho}
\|\rho^\e - \rho\|_{L^\infty([0,T^*];L^1_{\rm loc}(\R^N))} \to 0
\eq
as $\e \to 0$.

The strong convergence of $\rho^\e u^\e$ and $\rho^\e u^\e \otimes u^\e$ follows from the following standard estimates:
 \cite[Lemma 2.2]{CCJ21} 
\[
\begin{aligned}
&\|\rho^\e u^\e - \rho u \|_{L^1} \le \| \rho^\e \|_{L^1}^{\frac{1}{2}} \lt(\int \rho^\e |u^\e-u|^2\,\dx\rt)^{\frac{1}{2}} + \|u\|_{L^\infty}\|\rho^\e -\rho\|_{L^1},\\
&\|\rho^\e u^\e \otimes u^\e - \rho u \otimes u\|_{L^1}  \\
&\qquad\; \le \int \rho^\e|u^\e-u|^2 \,\dx +2\| \rho^\e\|_{L^1}^\frac{1}{2}\lt(\int \rho^\e|u^\e-u|^2 \,\dx \rt)^\frac{1}{2}\|u\|_{L^\infty} + \|\rho^\e-\rho\|_{L^1}\|u\|_{L^\infty}^2.
\end{aligned}
\]
Lastly, we proceed to establish the convergence of the pressure tensor as
\bq\label{bridge}
\int (v-u^\e)\otimes (v-u^\e)f^\e\,\dv - P_\psi(\rho)\bbI_N = \calR^\e +(P_\psi(\rho^\e)-P_\psi(\rho))\bbI_N,
\eq
where $\calR^\e$ is defined as in \eqref{eq:pe}. We have 
\bq\label{arch}
\|\calR^\e\|_{L^1([0,T^*];L^1(\R^N))} \le C_*\e^{\frac{1}{2}}
\eq
following the $J_4$ estimate in the proof of Lemma \ref{lem:eps}. We then notice that
\[
P_\psi(\rho^\e)-P_\psi(\rho) = P_\psi(\rho^\e|\rho) + h_\psi(\rho)(\rho^\e -\rho).
\]
Since $|P_\psi(\rho^\e|\rho)| \le C \Phi_\psi(\rho^\e|\rho) \le C \eta(U^\e_t|U_t)$ and $h_\psi(\rho) \in L^\infty([0,T^*]\times \R^N)$ with \eqref{crho}, we have
\[
\|P_\psi(\rho^\e) - P_\psi(\rho)\|_{L^\infty([0,T^*];L^1_{\rm loc}(\R^N))} \to 0.
\]
Combining these two, we obtain
\[
\lt\|\int (v-u^\e)\otimes (v-u^\e)f^\e\,\dv - P_\psi(\rho)\bbI_N \rt\|_{L^1_{\rm loc}([0,T^*]\times \R^N)} \to 0.
\]

This establishes Theorem \ref{thm:hdr}; the estimate of relative entropy  \eqref{mnre} is established in Lemma \ref{lem:eps}, the strong convergece of macroscopic quantities are addressed in Section \ref{ssc:strcv}.

\subsection{Convergence of macroscopic quantities: power law case}
In this section, we prove Corollary \ref{quant} to obtain a quantitative convergence rate of the case of $P_{\psi_m}(\rho)=\theta \rho^{\gamma(m)}$. We recall some useful inequalities. For $p>1$, we denote
\[
\varphi_p(f|g):= f^p - g^p - pg^{p-1}(f-g).
\]
for $f,g \ge 0$. Let us denote $\Omega$ be any subset of $\R^N$. We recall the following classical results.

\begin{lemma}\label{lem:ple2}
Let $f,g \in L^p_+(\Omega)$ be given. If $p \in (1,2]$, then we have
\[
\|f-g\|_{L^p(\Omega)} \le C_{p} \lt(\int_{\Omega} f^p + g^p \,\dx\rt)^{\frac{2-p}{2p}} \lt(\int_{\Omega} \varphi_p(f|g)\,\dx\rt)^{\frac{1}{2}}
\]
where $C_p$ is independent of $\Omega$.
\end{lemma}
\begin{proof}
Note that by Taylor's theorem, we have
\[
f^p - g^p - pg^{p-1}(f-g) \ge \frac{p(p-1)}{2}\min\lt\{\frac{1}{f},\frac{1}{g}\rt\}^{2-p}(f-g)^2.
\]
On the other hand, we observe that
\[
|f-g|^{p} = \max\{f,g\}^{\frac{p(2-p)}{2}} \min\lt\{\frac{1}{f}, \frac{1}{g}\rt\}^{\frac{p(2-p)}{2}}|f - g|^{p}.
\]
Applying H\"older's inequality with exponent $\lt(\frac{2}{p},\frac{2}{2-p}\rt)$, we obtain
\[
\begin{split}
\int_{\Omega} |f-g|^{p} \dx &\le \lt(\int_{\Omega} \max\{f,g\}^{p} \dx\rt)^{\frac{2-p}{2}}\lt(\int_{\Omega} \min\lt\{\frac{1}{f}, \frac{1}{g}\rt\}^{2-p} |\rho^\e - \rho|^2 \dx\rt)^{\frac{p}{2}} \\
&\le \lt(\frac{2}{p(p-1)}\rt)^{\frac{p}{2}} \lt(\int_{\Omega} f^p + g^p\, \dx\rt)^{\frac{2-p}{2}} \lt(\int_{\Omega} f^p - g^p - pg^{p-1}(f-g)\, \dx\rt)^{\frac{p}{2}}.
\end{split}
\]
\end{proof}

\begin{lemma}\label{lem:pge2}
For $p\ge 2$, then the following pointwise estimate holds:
\[
|b-a|^p \le c(p)\lt(|b|^p - |a|^p - p|a|^{p-2} a \cdot (b-a)\rt)\quad \forall a,b \in \R^d,
\]
for some $c(p)>0$. As a consequence, for any $f,g \in L^p_+(\Omega)$, we have
\[
\|f -g\|_{L^p(\Omega)} \le c(p)^{\frac{1}{p}} \lt(\int_{\Omega} \varphi_p(f|g)\,\dx\rt)^{\frac{1}{p}}. 
\]
\end{lemma}

By applying Lemma \ref{lem:ple2},  \ref{lem:pge2} with $p=\gamma$ and Lemma \ref{lem:eps}, we can easily obtain that
\bq\label{milan}
\|\rho^\e - \rho\|_{L^\infty([0,T^*];L^\gamma(\R^N))} \lesssim \e^{\frac{1}{2\max\{2,\gamma\}}}.
\eq

\begin{lemma}%\label{lem:rhou}
It holds that
\[
\|\rho^\e u^\e - \rho u\|_{L^\frac{2\gamma}{\gamma+1}} \le \|\rho^\e\|_{L^\gamma}\lt(\int \rho^\e|u^
\e -u|^2\,\dx\rt)^{\frac{1}{2}} + \|\rho^\e - \rho\|_{L^\gamma}\|u\|_{L^{\frac{2\gamma}{\gamma-1}}}.
\]
\end{lemma}
\begin{proof}
The proof is based on following observation and H\"older's inequality.
\[
\rho^\e u^\e - \rho u = (\rho^\e)^{\frac{1}{2}}(\rho^\e)^{\frac{1}{2}}(u^\e -u) + (\rho^\e-\rho)u
\]
\end{proof}
This Lemma, combined with previous estimate for $\|\rho^\e -\rho\|_{L^\gamma}$ in \eqref{milan}, yields
\bq\label{bologna}
\|\rho^\e u^\e -\rho u\|_{L^\infty([0,T^*];L^\frac{2\gamma}{\gamma+1}(\R^N))}\lesssim \e^{\frac{1}{2\max\{2,\gamma\}}}.
\eq

\begin{lemma}%\label{lem:rhouu}
We have
\[
\|\rho^\e u^\e \otimes u^\e - \rho u \otimes u\|_{L^1} \le \int \rho^\e|u^\e-u|^2\dx +
2\|\rho^\e u^\e - \rho u\|_{L^{\frac{2\gamma}{\gamma+1}}}\|u\|_{L^{\frac{2\gamma}{\gamma-1}}}
+ \|\rho^\e-\rho\|_{L^\gamma}\|u\|_{L^{\frac{2\gamma}{\gamma-1}}}^2.
\]
\end{lemma}
\begin{proof}
The proof is based on the following observation,
\[
\rho^\e u^\e \otimes u^\e - \rho u \otimes u = \rho^\e(u^\e-u) \otimes (u^\e - u) +  (\rho^\e u^\e - \rho u)\otimes u+  u\otimes (\rho^\e u^\e - \rho u)- (\rho^\e-\rho)u\otimes u.
\]
\end{proof}
Similarly, applying this lemma with \eqref{bologna}, we obtain
\bq\label{gaeta}
\|\rho^\e u^\e \otimes u^\e - \rho u \otimes u\|_{L^\infty([0,T^*];L^1(\R^N))} \lesssim \e^{\frac{1}{2\max\{2,\gamma\}}}
\eq

To obtain the error estimate for pressure tensor, we notice the following:
\begin{lemma}%\label{lem:Z}
For any $f,g \in L^\gamma_+(\R^N)$ with $\gamma>1$, we have
\[
\intr |f^\gamma - g^\gamma|\,\dx \le \gamma\|f-g\|_{L^\gamma}\lt(\intr f^\gamma + g^\gamma \,\dx\rt)^{\frac{\gamma-1}{\gamma}}.
\]
\end{lemma}
\begin{proof}
For any $a,b\ge 0$, we observe that
\[
|a^\gamma-b^\gamma| \le \gamma |a-b| \max\{a,b\}^{\gamma-1}
\]
Hence, by H\"older's inequality, we obtain
\[
\intr |f^\gamma-g^\gamma|\,\dv \le \gamma\|f-g\|_{L^\gamma(\R^N)} \lt(\intr \max\{f,g\}^\gamma \,\dv\rt)^{\frac{\gamma-1}{\gamma}}.
\]
\end{proof} 
Applying this lemma with \eqref{milan} provides
\[
\|(\rho^\e)^\gamma - (\rho)^\gamma\|_{L^\infty([0,T^*];L^1(\R^N))} \lesssim \e^{\frac{1}{2\max\{2,\gamma\}}}.
\]
We then recall \eqref{bridge} and \eqref{arch} to find that
\bq\label{rome}
\lt\|\int (v-u^\e)\otimes (v-u^\e)f^\e\,\dv - P_\psi(\rho)\bbI_N\rt\|_{L^1([0,T^*]\times\R^N)} \lesssim \e^{\frac{1}{2\max\{2,\gamma\}}}.
\eq
 We summarize that the quantitative error estimates are provided in \eqref{milan}--\eqref{bologna}--\eqref{gaeta}--\eqref{rome}.

\appendix

\section{Assumptions on $L_\psi$ for generalized Log Sobolev inequality}\label{app}

In \cite{CJMTU01}, the assumptions on the nonlinear diffusion law $L_\psi$ were imposed to obtain the generalized Log Sobolev inequality. They read:
\begin{itemize}
\item[(HF1)] $L_\psi: \R_0^+ \to \R$ is strictly increasing continuous function satisfying $L_\psi(0)=0$.
\item[(HF2)] $L_\psi|_{\R^+} \in \calC^3$.
\item[(HF3)] $\psi \in L_{\rm loc}^{1}([0,\infty))$, $\Phi'(s)=\psi(s)$ for all $s \in \R^+$.
\item[(HF4)] $L_\psi(s) \le \frac{d}{d-1}sL_\psi'(s)$ for all $s>0$.
\item[(HF5)] There exists $s_0 \in \R^+$ with $\psi'(s) \ge 0$ for all $s \in (0,s_0)$.
\end{itemize}

Indeed, (HF1)--(HF5) can be easily deduced from (A1) and (A4). (HF1), (HF3) and (HF5) are immediate from (A1). To see (HF4), we notice that
\[
L_\psi(s) = \int_0^s L_\psi'(u)\,\du \le \int_0^s L_\psi'(s)\,\du = s L_\psi(s)
\]
since $L_\psi''(s)=\psi'(s) + s\psi''(s) \ge 0$ by (A4). Lastly, as $L_\psi'''(s) = 2\psi''(s) + s\psi'''(s)$, $\psi \in \calC^3((0,\infty))$ implies (HF2).

\bibliographystyle{alpha}

\end{document}